\documentclass[11pt,a4paper,reqno]{amsart}
\usepackage{cite}
\usepackage{amsmath,amsthm,amsfonts,amssymb,bbm}
\usepackage{graphicx,psfrag}
\usepackage{hyperref}
\usepackage{enumerate}
\usepackage[utf8]{inputenc}
\usepackage[T1]{fontenc}
\usepackage{stmaryrd}
\usepackage{subcaption}

\numberwithin{equation}{section}

\newcommand{\R}{\mathbb{R}}
\newcommand{\N}{\mathbb{N}}

\newcommand{\Z}{\mathbb{Z}}

\newtheorem{prop}{Proposition}[section]
\newtheorem{theorem}[prop]{Theorem}
\newtheorem{coro}[prop]{Corollary}
\newtheorem{lem}[prop]{Lemma}

\newcommand{\Esp}{\mathbb{E}}
\newcommand{\Var}{\mathrm{Var}}

\newcommand{\eps}{\varepsilon}

\newtheorem{rem}{Remark}

\newtheorem{defi}{Definition}

\title{Hydrodynamic limit for a 2D interlaced particle process}
\author{Vincent Lerouvillois}
\address{Univ Lyon,  Universit\'e Claude Bernard Lyon 1,  UMR 5208, Institut Camille Jordan,  F-69622 Villeurbanne cedex, France} \email{lerouvillois@math.univ-lyon1.fr}
\author{Fabio Lucio Toninelli}
\address{Technical University of Vienna, Institut f\"ur Stochastik und Wirtschaftsmathematik, Wiedner Hauptstra{\ss}e 8-10, A-1040 Vienna, Austria} \email{fabio.toninelli@tuwien.ac.at}
\begin{document}

\begin{abstract}
	The Markov dynamics of interlaced particle arrays, introduced by
	A. Borodin and P. Ferrari in \cite{BF}, is a classical example of
	$(2+1)$-dimensional random growth model belonging to the so-called
	{\sl Anisotropic KPZ} universality class.  In
	\cite{legras2017hydrodynamic}, a hydrodynamic limit -- the
	convergence of the height profile, after space/time rescaling, to
	the solution of a deterministic Hamilton-Jacobi PDE with non-convex
	Hamiltonian -- was proven when either the initial profile is convex,
	or for small times, before the solution develops shocks. In the
	present work, we give a simpler proof, that works {\sl for all
		times and for all initial profiles for which the limit equation
		makes sense}. In particular, the convexity assumption is
	dropped. The main new idea is a  new viewpoint about "finite speed of 
	propagation" that allows to bypass the need
	of a-priori control of the interface gradients, or equivalently of
	inter-particle distances.
\end{abstract}

\maketitle

%\tableofcontents

%=======================================================================

\section{Introduction}

In this work, we study a $(2+1)$-dimensional stochastic growth model,
or equivalently an irreversible Markov process for a two-dimensional
system of interlaced particles which perform totally
asymmetric, unbounded jumps. This model was originally introduced by
A. Borodin and P. Ferrari in \cite{BF} (we will refer to it as Borodin-Ferrari dynamics) together with a larger class of
growth models that belong to the so-called Anisotropic KPZ
universality class \cite{wolf1991kinetic}; we refer to
\cite{legras2017hydrodynamic,MR3966870,BF} for a discussion of this topic
and for further references. Our focus here is not on interface
fluctuations but on the hydrodynamic limit (i.e. the law of large
numbers for the height profile $H(x,t)$). For models in the AKPZ
class, the rescaled height profile is conjectured to converge to the
viscosity solution of a non-linear Hamilton-Jacobi PDE
\cite{barles2013introduction}
\begin{eqnarray}
  \label{eq:HJPDE}
  \partial_t u+v(\nabla u)=0,
\end{eqnarray}
with non-convex Hamiltonian $v$. In fact, the feature that
distinguishes the AKPZ class from the usual KPZ class is that the
Hessian $D^2 v$ has negative determinant \cite{wolf1991kinetic}.  The absence of
convexity has an important consequence on the hydrodynamic
behavior. First of all, there is no Hopf-Lax formula for the solution
of \eqref{eq:HJPDE}. Moreover, there is no hope that the subadditivity
arguments developed in  \cite{MR1797390,MR1921440} apply, since they would
automatically yield a convex Hamiltonian. Let us recall that
the methods of \cite{MR1797390,MR1921440} require (at the microscopic level)
the so-called ``envelope property'', which is a strong version of monotonicity and which gives a microscopic analog of
Hopf-Lax formula. For AKPZ models, the envelope property simply fails
to hold.

In the previous work \cite{legras2017hydrodynamic}, the convergence
(in probability) of the height profile to the viscosity solution of
\eqref{eq:HJPDE} was proven, under the important restriction that
either the initial profile $u_0$ is convex, or that time is smaller
than the time when shocks (discontinuities of $\nabla u$) appear. In
the present article, instead, we prove the result \emph{for all
  initial profiles and for all times} (and convergence holds almost
surely).  Our proof is strongly inspired by the method developed by
F. Rezakhanlou in \cite{rezakhanlou2001continuum}; in few words, it
consists in showing that the Markov semigroup that encodes the
dynamics is tight in a certain topology and that all of its limit
points satisfy a set of properties that are sufficient to identify
them with the unique semi-group associated to the PDE
\eqref{eq:HJPDE}. These ideas have been recently employed in
\cite{zhang2018domino,lerouvillois2019hydrodynamic} to obtain a full
hydrodynamic limit for other $(2+1)$-dimensional growth models in the
AKPZ class, namely the ``domino shuffling algorithm'' and the
Gates-Westcott model \cite{MR1478060}. The reason why Rezakhanlou's
method was not employed in \cite{legras2017hydrodynamic} is that it
requires a strong (i.e. uniform with respect to the initial condition)
form of locality or of ``finite speed of propagation'' for the
dynamics. Such property is easy to check for the ``domino shuffling''
or the Gates-Westcott dynamic, but it fails for the Borodin-Ferrari
dynamics. The reason for this is that particle jumps are not bounded;
the larger the typical inter-particle distance in the initial
condition, the faster information propagates through the system.

The main new idea of the present work, that allows to overcome the
limitations of \cite{legras2017hydrodynamic}, is the following.  The
usual locality property would say that the height function $H(x,t)$,
for fixed $x$ and $t\le T$, is with high probability not influenced
(for $T$ large and \emph{uniformly in the initial condition}) by the
value of the initial profile $H(\cdot,0)$ outside a ball of radius
$O(T)$ centered at $x$. This fails for the Borodin-Ferrari dynamics:
locality holds but not uniformly, due to the lack of any a-priori
control of typical inter-particle distances at times $t>0$.  In
contrast, our Proposition \ref{prop:dlocalite} shows that the height
$H(x,t)$ is entirely determined by the height $H$ on a certain
deterministic, compact subset of space-time, depending on $(x,t)$ but
independent of the initial condition.  {As explained in Section
  \ref{sec:locality}, underlying this property is a bijection between the
  Borodin-Ferrari dynamics and a discretized version of the
  Gates-Westcott growth model.} This bijection is the analogue of the well-known mapping between the Hammersley process and the Polynuclear growh (PNG) growth model \cite{MR2249629}.

As a side remark, with respect to the method of \cite{rezakhanlou2001continuum}, we avoid studying directly the convergence of the Markov semi-group; this streamlines somewhat the argument.

{\bf Organization of the article.} This work is organized as follows.
The model and its height function are defined in Section \ref{sec:BF}
and \ref{sec:height}, while the main theorem is given in Section
\ref{sec:maintheo}. Section \ref{sec:properties} proves a few general
properties of the process, including the new locality statement.
Compactness of sequences of rescaled height profiles is proven in
Section \ref{sec:compactness} and the identification of limit points
with the solution of the PDE is obtained in Section
\ref{sec:identification}.

\section{The model and the main result}

\subsection{The Borodin-Ferrari dynamics}
\label{sec:BF}
We start by recalling the definition of the Borodin-Ferrari
dynamics, as a Markov evolution of a
two-dimensional array of interlaced particles. In the original
reference \cite{BF}, particles perform  jumps of
length 1 to the right and can ``push'' a number $n\ge 0$ of other particles; we
rather follow the equivalent representation used in
\cite{Toninelli2+1,legras2017hydrodynamic}, where particles jump a distance
$k\ge 1$ to the left, and no particles are pushed.  The lattice where
particles live consists of an infinite collection of discrete horizontal lines,
labeled by an index $\ell\in \mathbb Z$. Each line contains an
infinite collection of particles, each with a label $(p,\ell)$,
$p\in \mathbb Z$. See Figure \ref{fig:1}. Horizontal particle
positions $z_{(p,\ell)}$ are discrete:
\[
z_{(p,\ell)}\in\mathbb Z+(\ell\!\!\!\!\mod 2)/2
\] (note that adjacent horizontal lines are displaced by a half interger).  % For convenience, we let
% \[
% \mathbb Z_\ell:=\mathbb Z+(\ell\!\!\!\!\mod 2)/2.
% \]
% PAS TERRIBLE PEUT ETRE COMME NOTATION...

\begin{defi}
  \label{def:omega}We let $\Omega$ be the set of  particle configurations $\eta$ satisfying the following properties:
\begin{enumerate}
\item no two particles in the same line $\ell$ share the same position
  $z_{(p,\ell)}$. We label particles in each line in such a
  way that $z_{(p,\ell)}<z_{(p+1,\ell)}$. Labels are
  attached to particles, and they do not change along the dynamics.
\item particles are interlaced: for every
  $\ell$ and $p$, there exists a unique $p'\in \mathbb Z$ such that
  $z_{(p,\ell)}<z_{(p',\ell+1)}<z_{(p+1,\ell)}$ (and, as a
  consequence, also a unique $p''\in \mathbb Z$ such that
  $z_{(p,\ell)}<z_{(p'',\ell-1)}<z_{(p+1,\ell)}$). Without loss of
  generality, we  assume that $p'=p$ (and therefore
  $p''=p+1$). This can always be achieved by deciding which particle
  is labeled $0$ on each line.  Also, by convention, we establish that the particle labeled $(0,0)$ is the left-most one on line $\ell=0$, with non-negative  horizontal coordinate.
\item for  $\ell=0$ (and therefore for every $\ell$, because of the interlacement condition) one has
\begin{eqnarray}
  \label{eq:1}
 \lim_{p\to -\infty} \frac{z_{(p,\ell)}}{p^2}=0.
\end{eqnarray}
\end{enumerate}
\end{defi}

% \[F:=\{z\in E'|\forall l\in Z, \lim_{p\to \infty} \frac{z_{(p,\ell)}}{p^2}=0\} \]    
%C'EST UNE NOTATION PLUS STANDARD POUR LES CHAINES DE MARKOV, EN PLUS E PEUT ETRE CONFUS AVEC UNE ESPERANCE.

% $q:=(p,l)$ is said to be a \textit{particule label}, $z_q$ denotes the position of $p$-th particule of the $l$-th line.

Here we give an informal description of the dynamics; a rigorous
(graphical) definition was given in
\cite[Sec. 2.3]{legras2017hydrodynamic}. We will not recall the
details of this construction here and we will use it only implicitly:
we will work with the informal definition of the dynamics, and the
existence of the graphical construction guarantees that the arguments
are actually rigorous.

Given the particle label $(p,\ell)$, we denote $I_{(p,\ell)}=\{(p-1,\ell+1),(p,\ell-1)\}$, see Fig. \ref{fig:1}: 
these are  the labels of the two particles directly to the left of $(p,\ell)$
on lines $\ell+1$ and $\ell-1$.
\begin{figure}[htbp!]
\includegraphics[width=9cm]{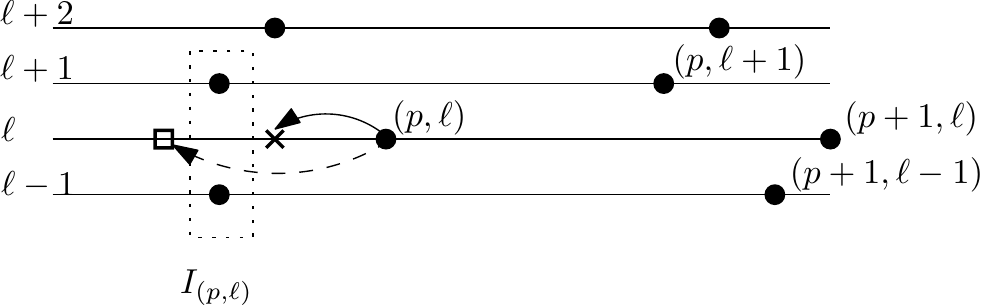}
  \caption{The set $I_{(p,\ell)}$ comprises the particle labels of the two particles in the dotted region. The particle $(p,\ell)$ can jump  to the position marked by the cross (and does so with rate $1$) but cannot jump to the position marked by a square.  }
\label{fig:1}
\end{figure}

  To every pair $(\ell,z)$ with
$\ell\in\mathbb Z$ and $z\in \mathbb Z+(\ell\!\!\mod 2)/2$ we
associate an i.i.d. Poisson clock of rate $1$ (we will denote $W$ the realization of all the Poisson processes).   When the clock labeled
$(\ell,z)$ rings, then:
\begin{itemize}
\item if position $(\ell,z)$ is occupied, i.e. if there is a particle
  on line $\ell$ with horizontal position $z$, then nothing happens;
\item if position $(\ell,z)$ is empty, let $(p,\ell)$ denote the label
  of the left-most particle on line $\ell$, with $z_{(p,\ell)}>z$. If
  both particles with label in $I_{(p,\ell)}$ have horizontal position smaller
  than $z$, then particle $(p,\ell)$ is moved to position $(\ell,z)$;
  otherwise, nothing happens.
\end{itemize}
In words: is position $(\ell,z)$ is empty, then particle $(p,\ell)$ is moved to position  $(\ell,z)$ if and only if the new configuration is still in $\Omega$, i.e. if the interlacement constraints are still satisfied. Despite the fact that particle jumps are unbounded, the dynamics is well defined for almost every realization $W$ of the Poisson clocks, thanks to the
 condition
\eqref{eq:1} on particle spacings. A proof via the graphical representation is given in \cite{legras2017hydrodynamic}.

As observed in \cite[Remark 2.2]{legras2017hydrodynamic}, the
evolution of the particles on each line $\ell $ follows the
one-dimensional (discrete) Hammersley-Aldous-Diaconis (HAD) process
\cite{FerrariMartin}, except that particle jumps can be prevented by
the interlacing constraints with particles in lines $\ell\pm1$. This
induces very strong correlations between the processes on different
lines.  As proven in \cite{Toninelli2+1}, the translation invariant,
stationary measures of the Borodin-Ferrari dynamics correspond (via
the tiling-to-interlacing particle bijection recalled in Section
\ref{sec:loz}) to the the translation invariant, ergodic Gibbs
measures of rhombus tilings of the triangular lattice.  In particular,
the restriction of the stationary measures to any line $\ell$ is very
different from the (i.i.d. Bernoulli) invariant measures of the HAD
process: while for the latter the particle occupation variables are
i.i.d., for the former they have power-law correlations.

\subsection{Height function}
\label{sec:height}
To each configuration $\eta\in\Omega$ we associate an integer-valued
height function $h_\eta$. The relation with the height function of rhombus tilings of the plane is recalled in Section \ref{sec:loz}.

The graph $G$ whose vertices are  all the possible particle positions $(\ell,z),\ell\in\mathbb Z, z\in \mathbb Z+(\ell\!\!\mod 2)/2$ and where the neighbors of $(\ell,z)$ are 
the four vertices $(\ell\pm1,z\pm1/2)$ can be identified with $\mathbb Z^2$, rotated by $\pi/4$ and suitably rescaled, see Figure \ref{fig:2}.
The height function $h_\eta$ is defined on the dual graph\footnote{with a minor 
abuse of notation, we will often write $x\in \mathbb Z^2$ instead of $x\in G^*$}
$G^*$, obtained by shifting $G$ horizontally by $1/2$, see Figure
\ref{fig:2}.
\begin{figure}[htbp!]
	\includegraphics[width=11cm]{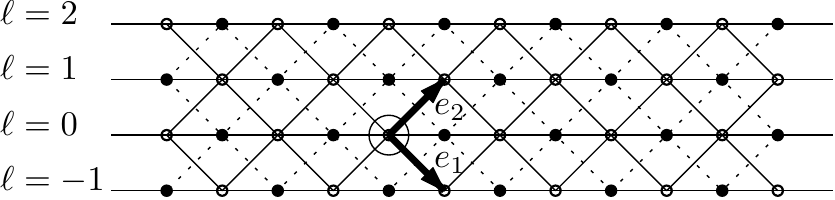}
	\caption{The lattice $G$ (dotted, with black vertices that are
		possible particle positions $(\ell,z)$, with
		$z\in\mathbb Z+(\ell\!\!\mod 2)/2$), the lattice $G^*$ (bold, with
		white vertices) and the two coordinate vectors $e_1,e_2$. The origin
		of $G^*$ (encircled) is the point of horizontal coordinate $z=-1/2$
		on line $\ell=0$.}
	\label{fig:2}
\end{figure}

We make the following choice
of coordinates on $G^*$:
\begin{defi}[Coordinates on $G^*$]
\label{def:coord}
 The point of $G^*$ of horizontal coordinate $-1/2$ of the line
labeled $\ell=0$ is assigned the coordinates $(x_1,x_2)=(0,0)$. The unit vector
$e_1$ (resp. $e_2$) is the vector from $(0,0)$ to the point of horizontal coordinate $0$ on the line labeled $\ell=-1$
(resp. $\ell=+1$), see Figure \ref{fig:2}. With this convention, the vertex of $G^*$ labeled $(x_1,x_2)$ is on line 
\begin{equation}
  \label{eq:elle}
\bar\ell(x)=x_2-x_1  
\end{equation}
 and has
horizontal coordinate 
\begin{equation}
  \label{eq:zeta}
\bar z(x)=(x_1+x_2-1)/2.   
\end{equation}
\end{defi}
We can now define the height function $h_\eta$:
% See Figure ... %%  Note that the collection  $G^*$ of
%% points where the height is defined is naturally seen as $\mathbb
%% Z^2$, turned by an angle $\pi/4$.
\begin{defi}[Height function]
\label{def:height}
Given a configuration $\eta\in\Omega$, its height function $h_\eta$ is
an integer-valued function defined on $G^*$.  We fix  $h_\eta(0,0)$ 
 to  any constant (for instance, zero). The  gradients
$h_\eta(x_1+1,x_2)-h_\eta(x_1,x_2)$ and
$h_\eta(x_2,x_2+1)-h_\eta(x_1,x_2)$ are defined as follows.
Given $(x_1,x_2)\in G^*$, let $p$ (resp. $p+1$) be the index of the rightmost
(resp. leftmost) particle on line $\ell=x_2-x_1$ that is to the left
(resp. to the right) of $(x_1,x_2)$. Recall that particle $p$ of line
$\ell+1$ satisfies $z_{p,\ell}<z_{p,\ell+1}<z_{p+1,\ell}$. We
establish that
\begin{multline}
  \label{eq:height}
\Delta_2 h_\eta(x_1,x_2):= h_\eta(x_1,x_2+1)-h_\eta(x_1,x_2)\\=\left\{
    \begin{array}{ccc}
      0 & \text{if} & (x_1,x_2+1) \text{ is to the right of particle } (p,\ell+1)\\
      1 & \text{if} & (x_1,x_2+1) \text{ is to the left of particle } (p,\ell+1)
    \end{array}
\right.
\end{multline}
and similarly 
\begin{multline}
  \label{eq:height2}
\Delta_1 h_\eta(x_1,x_2):=  h_\eta(x_1+1,x_2)-h_\eta(x_1,x_2)\\=\left\{
    \begin{array}{ccc}
      0 & \text{if} & (x_1+1,x_2) \text{ is to the right of particle } (p+1,\ell-1)\\
      1 & \text{if} & (x_1+1,x_2) \text{ is to the left of particle } (p+1,\ell-1).
    \end{array}
\right.
\end{multline}

\end{defi}
See Figure \ref{fig:3}.

It is immediately checked that\begin{multline}
  \label{eq:grad2}
\Delta_1 h_\eta(x_1,x_2)+\Delta_2 h_\eta(x_1+1,x_2)=\Delta_2 h_\eta(x_1,x_2)+\Delta_1 h_\eta(x_1,x_2+1)\\
=h_\eta(x_1+1,x_2+1)-h_\eta(x_1,x_2)\\=\left\{
    \begin{array}{ll}
      0 & \text{if $\exists$    particle between } (x_1,x_2) \text{ and } (x_1+1,x_2+1)\\
      1 & \text{if $\nexists$    particle between } (x_1,x_2) \text{ and } (x_1+1,x_2+1).
    \end{array}
\right.
\end{multline}
The first equality in \eqref{eq:grad2} implies that the sum of gradients of 
$h_\eta$ along any closed circuit is zero, so that the definition of $h_\eta$ 
is well-posed.
\begin{figure}[htbp!]
	\includegraphics[width=11cm]{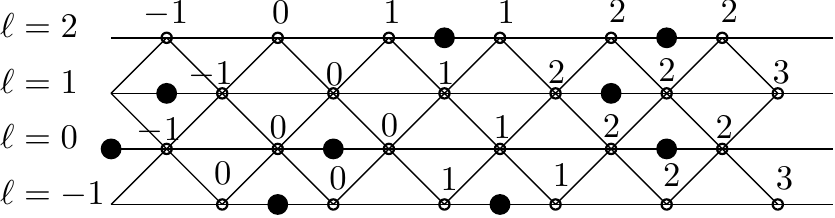}
	\caption{A portion of particle configuration and the corresponding height 
	function. Particles (black dots) are on vertices of $G$ and the height is 
	defined on vertices of $G^*$.}
	\label{fig:3}
\end{figure}

Let $\Gamma$ be the set of admissible height functions:
\begin{eqnarray}
  \label{eq:Gamma}
\Gamma := \{ h_{\eta}, \: \eta \in \Omega\},
\end{eqnarray}
where it is understood that we allow the value $h_\eta(0)$ to be any integer. 
 We also recall from \cite{legras2017hydrodynamic} the definition 
\begin{equation}
\Omega_M := \{ \eta \in \Omega, \: z_{(p,\ell)}-z_{(p-1,\ell)} \leq M   \text{ for every } \ell,p  \}
\end{equation}
and we introduce the corresponding set of height functions
\begin{equation}
\Gamma_M:= \{ h_{\eta}, \: \eta \in \Omega_M\}.
\end{equation}

Given an initial height function $ h \in \Gamma$, the height 
at time $t$ is defined as
\begin{eqnarray}
  \label{eq:Hxt}
  H(x,t)=h(x)-J_x(t),
\end{eqnarray}
where $J_x(t)\ge0$ denotes the number of particles, on the line 
labelled
$\bar \ell(x)=x_2-x_1$, that crossed (from right to left) the horizontal 
position $\bar z(x)=(x_1+x_2-1)/2$
in the time interval $[0,t]$.

  To emphasize the dependence of $H$
on the initial height configuration $h$ and on the
realization $W$ of the Poisson Point Process of intensity $1$ on
$G \times \R_+$, we write more explicitly
\[ H(x,t;h,W).\]

\subsubsection{Height function and mapping to rhombus tilings}
\label{sec:loz}

For a better intuition on the  height function,
let us recall that there is a bijection between
interlaced particle configurations satisfying properties (1)-(2)
of Definition \ref{def:omega} and rhombus tilings of the plane, as in
Figure \ref{fig:4}. 
\begin{figure}[htbp!]
\includegraphics[width=10cm]{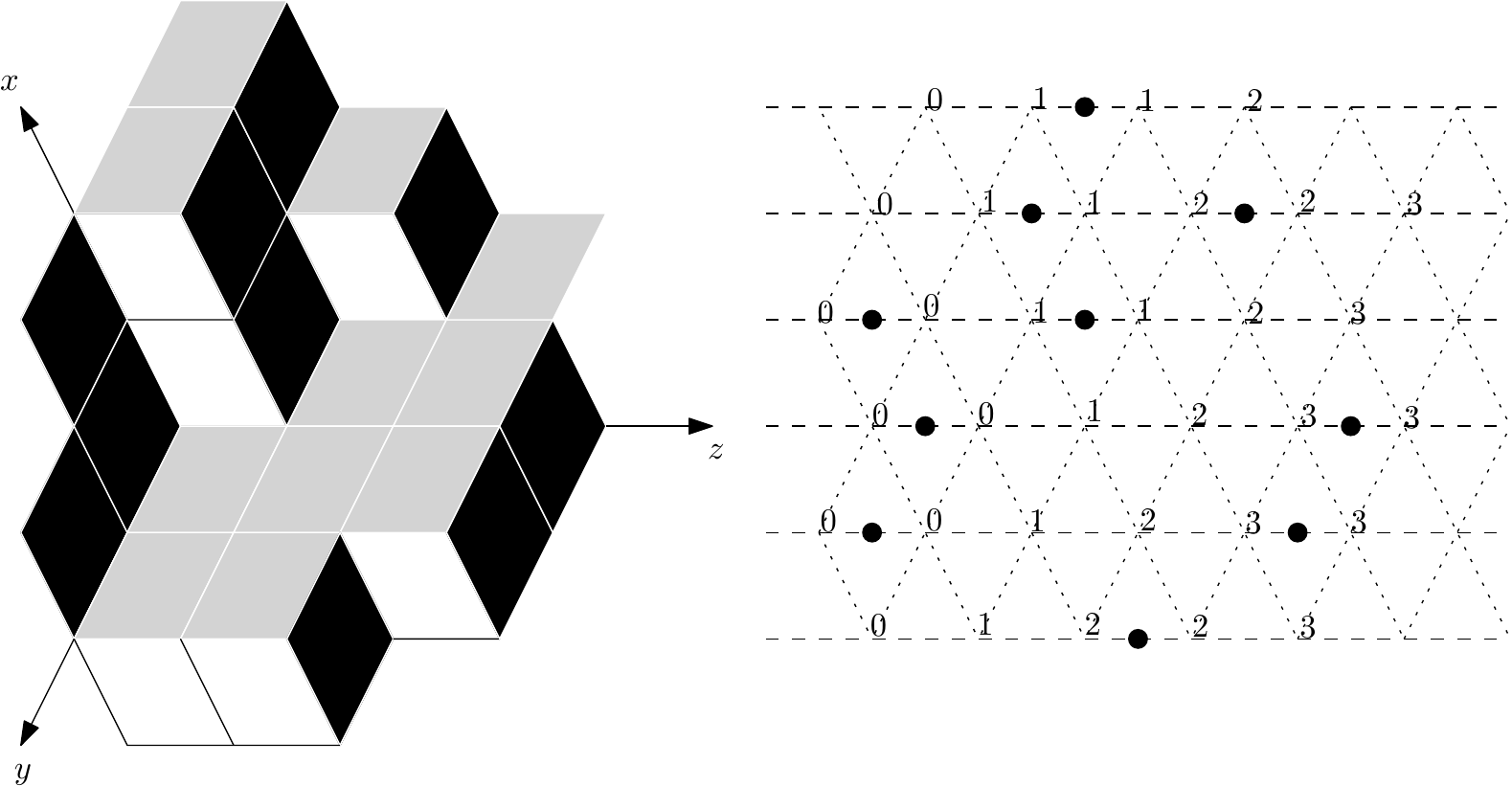}
\caption{The bijection between rhombus tiling (or stepped interface)
  and particle configuration: vertical (black) rhombi correspond to
  particles (dots). It is easily checked that the height function in
  the right figure corresponds to the height in the $z$ direction, that is  
  w.r.t. the $(x,y)$ plane,
  of the pile of cubes in the left figure. }
\label{fig:4}
\end{figure}\\
Particles correspond to vertical rhombi: the vertical coordinate of
the centre of a vertical rhombus defines the line the particle
is on, and its horizontal coordinate corresponds to the $z_{(p,\ell)}$
coordinate of the particle. If lengths are rescaled in such a way that
rhombi have sides of length $1$, then horizontal positions 
are shifted by half-integers between neighboring lines (as is the case
for particles). It is well known (and easy to understand from the
picture) that horizontal positions of rhombi in neighboring lines
satisfy the same interlacing conditions as particle positions $z_{(p,\ell)}$,
and  that the tiling-to-particle configuration mapping is a
bijection.

Given a rhombus tiling as in Figure \ref{fig:4} and viewing it as the
boundary of a stacking of unit cubes in $\mathbb R^3$, a natural
definition of height function is to assign to each vertex of a rhombus
the height (i.e. the $z$ coordinate) w.r.t. the $(x,y)$ plane of the
point (in $\mathbb R^3$) in the corresponding unit cube. As a
consequence, height is integer-valued and defined on points that are
horizontally shifted $1/2$ w.r.t. centers of rhombi, i.e. on points of
$G^*$.  The height function defined in the previous section
equals  the height of the stack of cubes w.r.t. the
$(x,y)$ plane.

\subsection{Slopes and speed}

Here we define the set of continuous height functions that are
possible scaling limits of the discrete height profile $H$, and the
speed of growth function $v$ (or ``Hamiltonian'') that appears in the limit PDE. Given an integer $M$, define
$\bar{\Gamma}_M$ to be the set of functions $f : \R^2 \rightarrow \R$ that are
non-decreasing in both coordinates and such that
\[
\frac{f(x+\lambda(1,1))-f(x)}{\lambda} \leq 1-\frac{1}{M} \text{ for every 
} 
x\in 
\R^2, \, \lambda \in \R.\]
 If $f$ is differentiable, this means that  $\nabla f$ belongs to the triangle 
$\mathbb{T}_M$ defined by
 \begin{eqnarray}
\label{eq:T_M}
\mathbb T_M:=\{\rho\in \R^2: \rho_1,\rho_2 \geq 0, \, \rho_1+\rho_2\leq 
1-M^{-1}\}.
\end{eqnarray}
Define also
\begin{equation}
\bar{\Gamma}:= \bigcup_{M \in \N} \bar{\Gamma}_M \quad \text{and} \quad  
\mathbb T:= 
\{\rho\in \R^2: \rho_1,\rho_2 \geq 0, \, 
\rho_1+\rho_2<1\}=\bigcup_{M\in \N} \mathbb T_M.
\end{equation}

We define the speed function $v(\cdot)$ as follows:
\begin{equation}\label{eq:vitesse}
v(\rho) =  \frac{1}{\pi} \frac{\sin( \pi \rho_1) \sin(\pi 
  \rho_2)}{\sin(\pi(\rho_1+\rho_2))}\ge0, \quad \rho\in \mathbb T
    \end{equation}
    with the convention that $v(0)=0$.
    Note that $v(\cdot)$ vanishes if $\rho\in \mathbb T$ with 
    $\min(\rho_1,\rho_2)=0$.
    Also, $v(\cdot)$ tends  to $+\infty$ if $\rho\to(\bar\rho_1,1-\bar \rho_1)$ 
    with $ \bar\rho_1\in (0,1)$. On the other hand, $v(\cdot)$ does not admit a 
    unique limit for $\rho\to(0,1)$ or $\rho\to(1,0)$ (any value in $[0,+\infty]$ can be obtained as limit point).
    
    \begin{rem}
    \label{rem:nablavrho}
    The speed function $v$ is non-decreasing in both coordinates and
    is Lipschitz on every $\mathbb{T}_M$ since
    $ \nabla v(\rho) = \left(\frac{\sin^2(\pi \rho_2)}{\sin^2(\pi
        (\rho_1+\rho_2))},\frac{\sin^2(\pi \rho_1)}{\sin^2(\pi
        (\rho_1+\rho_2))} \right)$ is bounded on $\mathbb T_M$.  One
    can also check that the determinant of the Hessian of $v$ is
    negative (strictly negative in the interior of $\mathbb T$) and
    thus the model belongs to the AKPZ universality class
    \cite{wolf1991kinetic}.
    \end{rem}

  \subsection{The hydrodynamic limit}
\label{sec:maintheo}
  
%\subsection{Limite hydrodynamique}
\begin{theorem}\label{theo:principalLJ}
  Given  an integer
  $M$, let $f\in \bar{\Gamma}_M$ and let $(h_L)_{L \in \N} \in 
  \Gamma_{M}^{\N}$ be a sequence of height functions approaching $f $ in the 
  following sense:
\begin{equation}\label{eq:hypcondinitialLJ}
\forall R>0 \qquad \sup_{\|x\|\leq R} \left| \frac{1}{L} h_L(\lfloor Lx \rfloor) - f(x) \right| \underset{L \to \infty}{\longrightarrow} 0.
\end{equation}
  Then, for almost every realization $W$, the following hydrodynamic limit holds:
% \note{Doit-on imposer une condition initiale plus stricte comme en 
%\eqref{eq:condini} ? Ou au moins une condition uniforme en $L$ sur la pente a 
%l'infini comme $\lim_{x \to 
%-\infty}\sup_{L}\frac{L|x|-h_L(xL,xL)}{\sqrt{L|x|}} 
%= +\infty.$ Par exemple, on met des particules partout a partir d'une distance 
%d'ordre $L$ de l'origine.} \FT{Maybe it is enough to require (8), plus the 
%fact 
%that the initial condition belongs to $\Omega_M$ for some finite $M$ 
%independent of $L$? CHECK}
\begin{equation}\label{eq:limitehydroGW}
\forall T>0 \quad \forall R>0 \qquad  \sup_{\|x\| \leq R, t \in [0,T]} 
\left| \frac{1}{L} H(\lfloor Lx \rfloor , Lt;h_L,W) - u(x,t) 
\right| \underset{L \to \infty}{\longrightarrow} 0 \ ,
\end{equation}
where $u$ is the unique viscosity solution of the Hamilton-Jacobi equation:
\begin{equation}\label{eq:visc}
\left\{
\begin{aligned}
 \partial_t u + {v}(\nabla u) & = 0 \\
 u(\cdot,0) & = f.
\end{aligned}
\right.
\end{equation}
\end{theorem}

\begin{rem}
  % The integers $M$ and $M_0$ in the Theorem do not need to be
  % equal.
  The condition $h_L \in \Gamma_{M}$ for some integer $M$ makes sure
  that the dynamics is well defined and satisfies the "finite speed of
  propagation" property (see Proposition \ref{prop:finitespeed}); this
  condition could be somewhat weakened. On the other hand, the
  requirement that $f \in \bar{\Gamma}_{M_0}$ for some integer $M_0$
  (we take for simplicity $M_0=M$) ensures that the slopes remain
  uniformly away from the the side $\rho_1+\rho_2=1$ of the triangle
  $\mathbb T$ where the speed $v$ is ill-defined. This condition is in
  a sense optimal: in fact, if $f$ is for instance the affine function
  of slope $\rho$ with $\rho_1+\rho_2=1$ and $h_L$ approaches $f$ as
  in \eqref{eq:hypcondinitialLJ}, then the limit height profile will
  be either $+\infty$ for all positive times (if $\rho_1\in (0,1)$) or
  the limit is not necessarily unique (if $\rho_1\in\{0,1\}$), i.e. it may 
  depend on the microscopic details of the initial condition $h_L$.
\end{rem}

  \begin{rem}
    \label{rem:compar}
    As observed above, the function $v(\cdot)$ cannot be extented
    continuously to the whole boundary of $\mathbb T$, so
    Eq. \eqref{eq:visc} requires some care. What is really meant in
    the theorem is that $u$ is the unique viscosity solution of the
    PDE where $v(\cdot)$ is replaced by $\tilde{v}(\cdot)$, which is
    any Lipschitz extension of $v(\cdot)$ to the whole $\mathbb R^2$
    that coincides with $v$  on $\mathbb T_{M}$.  Since in this case the Hamiltonian is Lipschitz and depends only on the gradient, the standard theory
    of viscosity solutions implies that the solution $u$ exists and is
    unique, and the comparison principle  shows that if $f\in\bar\Gamma_{M}$, then
    $u(\cdot,t)\in \bar\Gamma_M$ for all times, so that $u$ does not
    depend on the way $\tilde v$ is defined outside $\mathbb T_M$. For a reference on Hamilton-Jacobi equations, see e.g \cite[Sections 5 and 7]{barles2013introduction}.
\end{rem}

\section{Properties of the microscopic dynamic}

\label{sec:properties}
In this section, we recall some basic properties of the dynamics
following \cite{legras2017hydrodynamic}. In addition, we prove in
Section \ref{sec:locality} a new locality property that is crucial in
the proof of Theorem \ref{theo:principalLJ}.

\subsection{Translation invariance, monotonicity and speed of propagation}
We begin with a couple of easy facts:
\begin{prop}\label{prop:micro}
The dynamics satisfies the following properties:
\begin{enumerate}
\item \emph{Vertical translation invariance}: $H(x,t;h+c,W) = H(x,t;h,W) +c$ 
for all $c \in \Z$
\item \emph{Monotonicity}: $h \leq h' \Rightarrow H(x,t;h,W) \leq H(x,t;h',W)$.
% \item \emph{Semi-group}: $H(x,t;h,W) = H(x,t-s;H(\cdot,s;h,W),\theta_s W)$ 
% where $\theta_s$ is the time-translation by $s$.
\end{enumerate}
\end{prop}
The former statement is trivial and the latter follows from \cite[Th. 5.7]{legras2017hydrodynamic}.

Next, we recall a locality property established in
\cite{legras2017hydrodynamic} and we improve it 
 to an almost sure statement (but the really new locality result
will come in next section). Informally, Proposition 5.12 in
\cite{legras2017hydrodynamic} tells that if we fix a lattice site $x$ and an initial height
function $h \in \Gamma_M$, then with probability $1-O(\exp(-c(M)n)$
the height function at $x$ up to time $n$ only depends on the Poisson
clocks at positions within distance $\mathrm{O}(n)$ from
$x$. % Essentially, the proof
% relies on controlling the probability that there exists a chain of length of 
% order $n$ of Poisson clocks ring in a time interval of  order $n$ (see 
% \cite[Proposition 4.5]{legras2017hydrodynamic}).
Without much extra effort (we leave details to the reader), it is
possible to show the following slightly stronger statement (the
difference with respect to \cite[Prop. 5.12]{legras2017hydrodynamic}
is that the claim \eqref{eq:tildeW} holds simultaneously for every
$h\in \Gamma_M$):
\begin{lem}
\label{lem:dlocal} Fix $x\in \Z^2$.
For every integer $M$, there exist $c=c(M)>0$ and 
$\Delta=\Delta(M)>0$ such that for every integer $n$, for a set of realizations 
$W$ of probability $1-ce^{-n/c}$ of the Poisson process,  the following happens 
for any $h 
\in \Gamma_M$:
\begin{equation}\label{eq:tildeW}
 H(x,t;h, W) = H(x,t;h,\tilde{ W}) \text{ for every } t \leq \Delta n,
 \end{equation}
 for every $\tilde{ W}$ that coincides with $ W$ on\footnote{Here, we
   are viewing $W$ as a locally finite subset of points of $G\times \mathbb R_{+}$, where the first 
    coordinate corresponds to 
    position of sites where the clocks ring, and the
   second coordinate corresponds to the time when they ring.}
\[
R_n(x) := B(x,2n) \times [0,\Delta 
n],
\]
with $B(x,r)$ the  ball of radius $r$ centered at $x$.
% $\bar{\ell}=\bar{\ell}(x)=x_2-x_1$ and $\bar{z}=\bar{z}(x)=(x_1+x_2-1)/2$ as in \eqref{eq:elle}, \eqref{eq:zeta}.
\end{lem}

% Let us define  for all $R \geq 0$ and $x_0 \in \Z^2$ the domain
% \begin{equation}\label{eq:DR}
% D(x_0,R) := \{ x \in \Z^2  \,: |x-x_0|\leq R% |\bar{\ell}(x)-\bar{\ell}(x_0)| \leq R \text{ and } 
% % |\bar{z}(x)-\bar{z}(x_0)| \leq R 
% \}
% \end{equation}
% \FT{and, for lightness of notation, $D(R):=D(0,R)$.}

In the proof of Theorem \ref{theo:principalLJ}, we need instead an almost sure result:
\begin{prop}[Finite speed of propagation]
	\label{prop:finitespeed}
For almost every $W$ the following holds: for every integer $M $ there exists 
$\alpha = \alpha(M)$ such that for every time $T  > 0$, every $x_0\in \R^2$ and 
every $R \geq 0$, we have for $L$ large enough
\begin{equation}\label{eq:tildeWbis}
 H(x,t;h, W) = H(x,t;h,\tilde{ W}) \text{ for every } t \leq LT, \, x\in 
 B(Lx_0,LR), \, h\in \Gamma_M
 \end{equation}
 for every $\tilde{ W}$ that coincides with $ W$ on
  \begin{equation}\label{eq:RL}
B(Lx_0,L (R+\alpha T)) \times 
  [0,LT].
  \end{equation}
  % \begin{equation}\label{eq:RL}
  % \mathcal{R}_L(x_L,R,T) := (\bar{z}(x_L),\bar{\ell}(x_L),0)+[-L(R+\alpha 
  % T),L(R+\alpha T)]^2 \times 
  % [0,LT].
  % \end{equation}
\end{prop}

\begin{proof}
  If we apply Lemma \ref{lem:dlocal} with
  $n := \lceil LT/\Delta \rceil$, use an union bound for all $x$ in
  $B(0,L^2)$ (so that it includes all $x$ in $B(Lx_0,LR)$ for any fixed $x_0 
  \in \R^2$ and all $L$
  large enough) and apply Borel-Cantelli Lemma, we get the following. For almost
  every $W$, every integer $M$, every rational $T > 0$, every $x_0 \in \R^2$ 
  and every 
  $R \geq 0$, \eqref{eq:tildeWbis} holds all $L$ large enough and
  for every $\tilde{ W}$ that coincides with $W$ on the domain
  $B(Lx_0,L (R+\alpha T)) \times [0,LT]$ with $\alpha$ chosen large
  enough ($\alpha= 2 \Delta^{-1}$ suffices) such that this domain
  contains every $R_{\lceil LT/\Delta \rceil}(x)$ with
  $x \in B(Lx_0,LR)$.  The rest of the proof follows from rational
  approximation of~$T$.
\end{proof}

\begin{coro}[Weak Locality]
\label{coro:finitespeed}
For almost every $W$ the following holds: for every integer $M$  there exists  
$\alpha = \alpha(M)>0$ such that for every $T> 0$,
every $R \geq 0$, for $L$ large enough
\begin{eqnarray}\label{eq:tildeWbisbis}
 \sup_{\substack{x \in B(0,LR)\\ t \leq LT}} |H(x,t;h, W) - H(x,t;h',W)| \leq 
 \sup_{x \in 
   B(0,L(R+\alpha T))} |h(x) - h'(x)|
\end{eqnarray}
 for every $h,h'\in\Gamma_M$.
\end{coro}

\begin{proof}
We fix $W$ in the event of probability $1$ of Proposition 
\ref{prop:finitespeed}. Let $M\in \N ,T>0,R \geq 0$ and let $h,h' \in 
\Gamma_M$. We define $m$ as the supremum in the 
r.h.s of \eqref{eq:tildeWbisbis} and we set $h'' := h' + m$. Since $h \leq h''$ on 
$B(0,L(R+\alpha T))$, the local version of monotonicity stated in 
\cite[Theorem 5.10]{legras2017hydrodynamic} implies that for all $L \geq 1$
\[
H(x,t;h, \tilde{W}) \leq H(x,t;h'',\tilde{W}) \text{ for every } x \in 
B(0,LR), \, t \leq LT
\]
where $\tilde{W}$ is the restriction of $W$ to $B(0,L (R+\alpha T)) \times 
[0,LT]$. By finite speed of propagation (Proposition
\ref{prop:finitespeed}) and by vertical translation invariance
(Proposition \ref{prop:micro}), for all $L$ large enough, we have for
all $t \leq LT$ and $x \in B(0,LR)$,
\begin{align*}
H(x,t;h,W) &= H(x,t;h, \tilde{W}) \leq H(x,t;h'',\tilde{W}) \\
& = H(x,t;h'+m,W) = H(x,t;h',W) +m.
\end{align*}
Similarly, we can show that for all $L$ large enough, for all $t \leq LT$ 
and $x \in B(0,LR)$,
$
H(x,t;h',W) \leq H(x,t;h,W) + m,
$
which concludes the proof of Corollary \ref{coro:finitespeed}.
\end{proof}

We point out that the speed of propagation $\alpha(M)$ is not uniform
in $M$.  This is a problem, since later we would need  to apply the
propagation of information result from time $s>0$ to time $t>s$ but,
while we know by construction that the initial condition belongs to
$\Gamma_M$, the same does not hold at the later times $s>0$ (and we do
not see how to obtain an apriori control on $M(s)$ such that
$\eta(s)\in\Gamma_{M(s)}$). This is the main difficulty that prevented
a proof of a full hydrodynamic limit in \cite{legras2017hydrodynamic}
and this is why we call a result like Corollary \ref{coro:finitespeed}
``weak locality'' (it is much weaker than the locality statements used
e.g. in
\cite{lerouvillois2019hydrodynamic,rezakhanlou2001continuum,zhang2018domino}
to prove full hydrodynamic limits). One of the novelties of this
article is a smarter locality property that we state and prove in the
next section.

\subsection{A new version of locality}
\label{sec:locality}
In this section, we prove a very useful result of locality that says informally 
that if any height function $H$ is smaller than a height function $H'$ in a 
specific bounded \emph{space-time} domain, 
then we also have that $H(x,t) \leq H'(x,t)$.

Let us start by introducing some notations.  For any $\ell \in \N$, we
note $T_{\ell}\subseteq \Z^2$ the following triangle of size $\ell$
(see Figure \ref{fig:triangle})
\begin{equation}
	T_{\ell} = \{ y=(y_1,y_2) \in \Z^2, \: -2\ell<y_1 \leq 0, \, -2\ell<y_2 
	\leq 0, \, -2\ell \leq y_1 + 
	y_2 \leq 0 \}
	\end{equation}
and for all $t\geq 0$,
\begin{equation}
	E_{\ell} = (0,0,-\ell) \cup  \bigcup_{y \in T_{\ell}\setminus\{(0,0)\}} \{y\} 
	\times [-\ell 
	-\lfloor(y_1 + y_2)/2 \rfloor-1 , -\ell - \lfloor(y_1 + y_2)/2 \rfloor]  
	\subseteq \Z^2 \times \R.
	\end{equation}

Finally, for $x \in \Z^2$ and $t \geq \ell$, we define the space-time translated set
\begin{equation}
	E_{x,t,\ell} := (x,t) + E_{\ell} \subseteq \Z^2 \times \R_+.
	\end{equation}

\begin{figure}[htbp!]
	\begin{center}
		\includegraphics[scale=0.5]{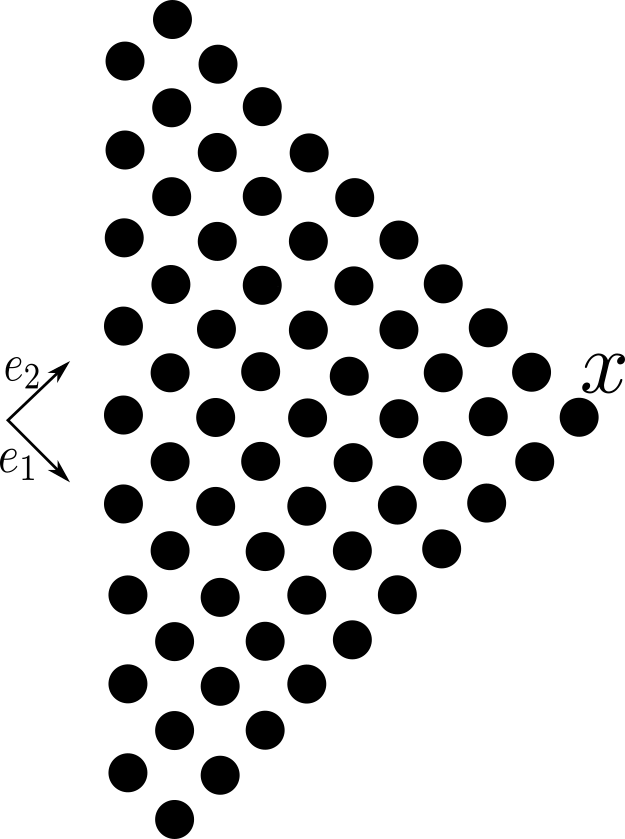}
		\caption{ \footnotesize The triangle $T_{\ell}$ for $\ell=5$.}
		\label{fig:triangle}
	\end{center}
\end{figure}

\begin{prop}[Space-time locality]\label{prop:dlocalite}
  Let $(x,t) \in \Z^2 \times \R_+$ and $l \in \N$ such that
  $t \geq \ell$. Also, let $h,h' \in \Gamma$. Then, we have
		\begin{eqnarray}
		\label{eq:triang}
		 H(\cdot,\cdot;h,W) \leq H(\cdot,\cdot;h',W) \text{ on } 
		E_{x,t,\ell} \Rightarrow  H(x,\cdot;h,W) \leq H(x,\cdot;h',W) 
		\text{ on } [t-\ell,t] .    
		\end{eqnarray}
\end{prop}

We emphasize that the main difference between a statement like
Corollary \ref{coro:finitespeed} and Proposition \ref{prop:dlocalite}
is that, while in both cases the difference between $H$ and $H'$ at some
 $(x,t)$ is estimated in terms of the difference in some domain,
in the former case the size of the domain depends on the bound  $M$ on 
the inter-particle distances, while in the latter it is independent
of it.

Before proving  Proposition
\ref{prop:dlocalite},
let us explain briefly where it comes from. This is best understood in the easier
case of the one-dimensional Hammersley (or Hammersley-Aldous-Diaconis)
process \cite{MR1355056}, whose definition we recall informally. The state
space consists of locally finite sets of points (or ``particles'')
that lie on $\R$. Particles jump to the left and jumps are determined
by a Poisson Point Process on $\R \times \R_+$ of intensity $1$: when
a clock rings at $(x,t)$, the leftmost particle to the right of $x$
jumps to $x$. The space-time trajectories of the particles can be
represented by a sequence of up-left paths as shown in Figure
\ref{fig}.
\begin{figure}[htbp!]
\includegraphics[width=9cm]{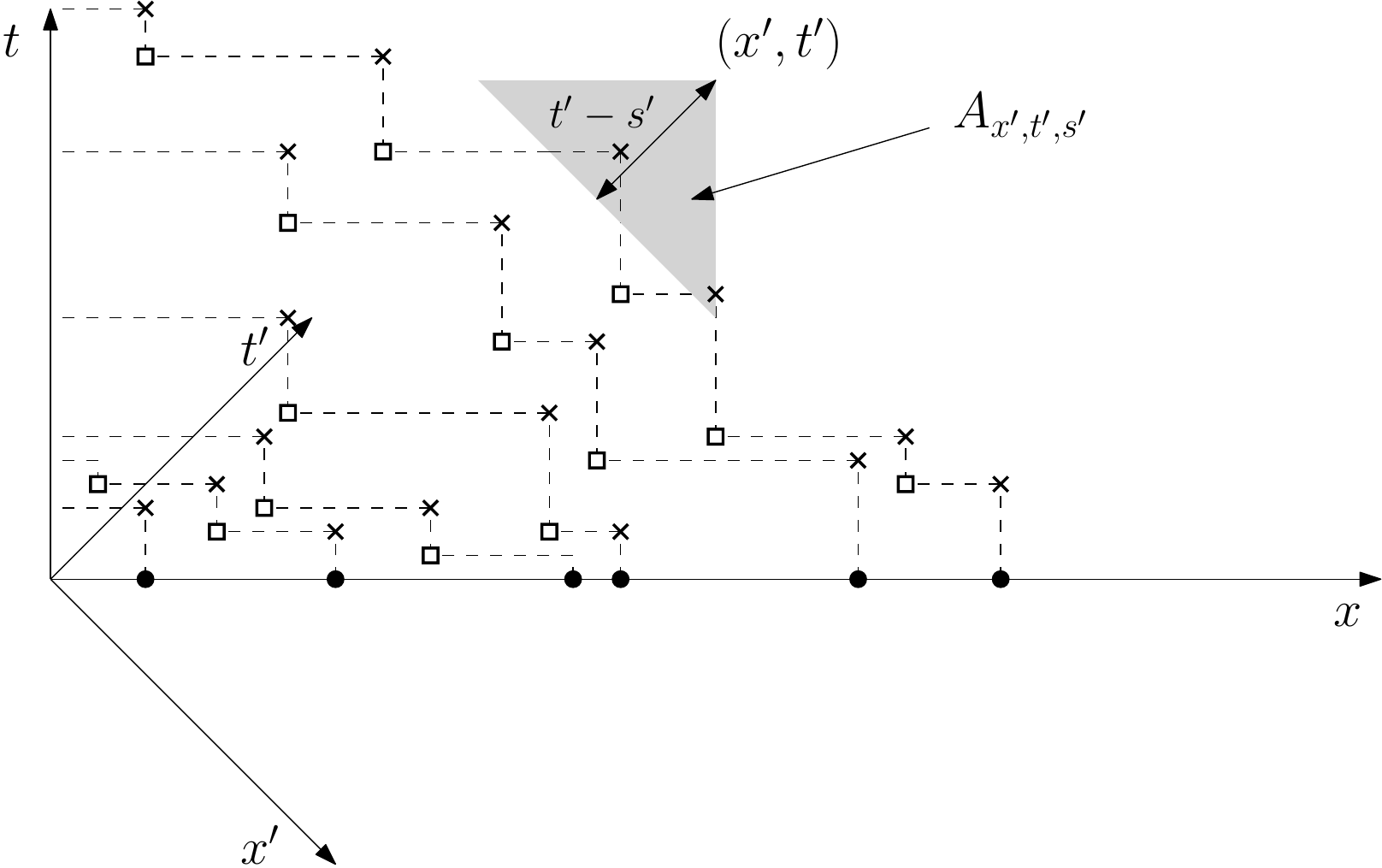}
\caption{Black dots indicate the position of particles at time zero
  for the Hammersley process, and dashed lines are their space-time
  trajectories. Poisson clock rings are marked with squares. Squares
  (resp. crosses) correspond to creations (resp. annihilations) of
  kink-antikink pairs for the PNG.%  The height difference (for the PNG, and therefore for the
  % Hammersley process) between
  % $(x_1,t_1)$ and $(x_2,t_2)$ with $x_2> x_1, t_2> t_1$ is given by
  % the number of dashed lines separating $(x_1,t_1)$ from
  % $(x_2,t_2)$.
  The height  at the point $(x',t')$ is determined by the
  height on the diagonal side of $A_{x',t',s'}$ and by the Poisson
  points inside the triangle.}
\label{fig}
\end{figure}

Rotating the picture by $45$ degrees, one obtains the
graphical representation of the so-called Polynuclear Growth Model
\cite{MR2249629}. In fact, it is well know that the Hammersley
process and the PNG are in bijection. Note however that the initial
condition at $t=0$ for the Hammersley process translates into a
condition on the line $t'=x'$ for the PNG (the coordinates $x',t'$ are as
in Figure \ref{fig}). For the PNG, a deterministic locality property
holds, because ``kinks'' and ``anti-kinks'' travel at deterministic speed
$1$. In other words, the PNG height function at $(x',t')$ is entirely
determined by  the Poisson points in the triangular
region $A_{x',t',s'}$ of Fig. \ref{fig} (with $s'<t'$) together with the height at time  $s'$ on a segment
$[x'-(t'-s'),x'+(t'-s')]$ (i.e. on the diagonal side of the triangle). By the bijection, we obtain
the same property for the Hammersley process and the diagonal side of
$A_{x',t',s'}$ is the analog of the set $E_{x,t,\ell}$ of Proposition
\ref{prop:dlocalite}. The reason why Proposition \ref{prop:dlocalite}
holds is that, also for the Borodin-Ferrari dynamics, one can establish a
bijection with a discretized version of the Gates-Westcott growth
model, for which information travels ballistically. The bijection with the Borodin Ferrari dynamics is not discussed explicitly in the literature and we do not 
describe it here either, as we need only its consequence,
Proposition \ref{prop:dlocalite}, of which we give a self-contained proof.

\begin{proof}[Proof of Proposition \ref{prop:dlocalite}]
	We proceed by induction and we assume without loss of generality that $x=0$. Let us start with the case $\ell=1$ (see Figure \ref{fig:triangle2}).
	
	For the sake of simplicity, let us write $H(\cdot,t)$ for $H(\cdot,t;h,W)$ and $H'(\cdot,t)$ for $H(\cdot,t;h',W)$. We 
	want to show:
	\begin{equation}\label{eq:casl=1}
	 H(\cdot,\cdot) \leq H'(\cdot,\cdot) \text{ on } E_{0,t,1} 
	 \Rightarrow  H(0,\cdot) \leq H'(0,\cdot) \text{ on } [t-1,t] \: .
	\end{equation}
	
	\begin{figure}[htbp!]
		\begin{center}
			\includegraphics[scale=.7]{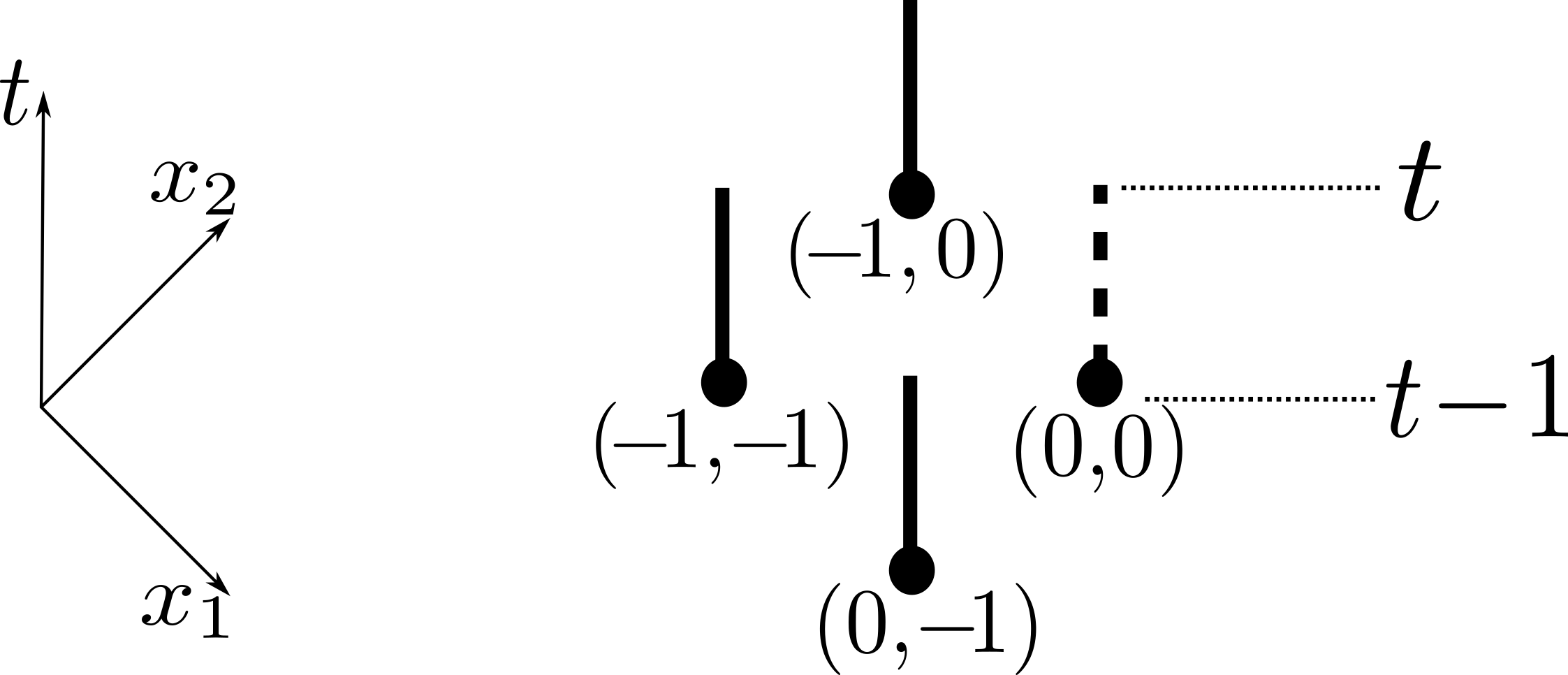}
			\caption{ \footnotesize Case $\ell=1$ of Proposition \ref{prop:dlocalite}. If $H \leq H'$ on the bold lines and points (corresponding to $E_{0,t,1}$), then 
			$H \leq H'$ also on the dash line (corresponding to $\{0\} \times (t-1,t]$).}
			\label{fig:triangle2}
		\end{center}
	\end{figure}
	
	%Modulo space-time translation and semi-group property, we can assume without loss of generality that $x=0$ et 
	%$t=1$. Let us assume that $H((y_1,y_2),t) \leq H'((y_1,y_2),t)$ for all 
	%$(y_1,y_2) \in \{(-1,-1),(-1,0),(0,-1)\}$ and all $t\in [0,1]$ and that 
	%$H((0,0),0) \leq H'((0,0),0)$. Let us show that $H((0,0),t) \leq H'((0,0),t)$ 
	%for all $t \in [0,1]$.
	
	Since height functions are almost surely right-continuous with
        $-1$ jumps (when a particle jumps), it is enough to show that
        if at time $t_0 \in [t-1,t]$, $H(0,t_0^-) = H'(0,t_0^-)$ and
        $H'$ decreases by $1$ at $0$, then it is also the case for
        $H$. By definition of the height function, it means that the
        left-most particle to the right of $0$ on line $0$ has crossed $0$ at 
        time
        $t_0$ for $H'$. We have to show that the same happens for
        $H$. Two cases can occur:
	\begin{enumerate}
        \item [(i)] The particle landed at the rightmost position on the left 
        of $x=0$, i.e. at horizontal position  $z=-1$ (corresponding to 
        position $(-1/2,-1/2)$ in the coordinates chosen on $G^*$).  By
          definition of the height function and of the dynamics, it is
          easily checked that the left-most particle to the right of
          $0$ can jump at position $z=-1$ (at time $t_0$, for the 
          configuration with height function $H'$) if and only if
		\[
		\left\{
		\begin{aligned}
		H'((0,-1),t_0^-) - 
				H'((-1,-1),t_0^-) &= 0 \\ H'((-1,0),t_0^-) - H'((-1,-1),t_0^-) &= 0 \\
		H'(0,t_0^-) - H'((-1,-1),t_0^-) &= 1.
		\end{aligned}
		\right.
		\]
		Since $H(0,t_0^-) = H'(0,t_0^-)$ and by assumption
		$H((-1,-1),t_0^-) \leq H'((-1,-1),t_0^-)$, we deduce that
		$$H(0,t_0^-) - H((-1,-1),t_0^-) = 1$$ (recall that from \eqref{eq:height}-\eqref{eq:height2} that discrete height gradients take values in $\{0,1\}$) and also 
		$H((-1,-1),t_0^-)=H'((-1,-1),t_0^-)$. Moreover,
                since $H((-1,0),t_0^-) \leq H'((-1,0),t_0^-)$ and
                $H((0,-1),t_0^-) \leq H'((0,-1),t_0^-)$, we have
                that
                $$H((-1,0),t_0^-) - H((-1,-1),t_0^-) =
                H((0,-1),t_0^-) - H((-1,-1),t_0^-) = 0$$ and thus
                the left-most particle to the right of $0$ for $H$ is also free to
                jump to position $z=-1$ at time $t_0$ and
                thus $H$ also decreases by $1$ at $0$.
		
		\item [(ii)]  The particle landed further to the left than $z=-1$.
		In this case, the height function $H'$ decreases by $1$ at time $t_0$ 
		at $0$ but also at $(-1,-1)$. Besides, this implies that there was no 
		particle at position $z=-1$ at time $t_0^-$. Therefore, we have
		\begin{align*}
		H(0,t_0) & \leq H((-1,-1),t_0) + 1 \\
		& \leq  H'((-1,-1),t_0) + 1 && \text{since $H \leq H'$ on $E_{0,t,1}$} \\
		& = H'((-1,-1),t_0^-) && \text{$H'$ decreases by $1$ at $(-1,-1)$ at 
		time $t_0$} \\
		& = H'(0,t_0^-) - 1 && \text{there was no particle at $z=-1$} \\
		& = H(0,t_0^-) - 1.
		\end{align*}
		Consequently, $H$ also decreased by $1$ at $0$ at time 
		$t_0$.
		
	\end{enumerate}
	
	Now let us show the inductive step, as illustrated in Fig. \ref{fig:cvr}: we assume that the
        result holds for some integer $\ell \geq 1$ and we show
        that it holds also for $\ell+1$. Let us assume that
        $H(.,.) \leq H'(.,.) \text{ on } E_{0,t,\ell+1}$. All
        we have to show is that this inequality is also true on
        $E_{0,t,\ell}$ and conclude by the induction hypothesis.
	
	To show this, we can apply the result  \eqref{eq:casl=1} for $\ell=1$
        to all $ (y,t-\ell - \frac{y_1+y_2}{2})$ with
        $y \in T_\ell$ such that $y_1+y_2$ is even and deduce that
        $H(y,\cdot) \leq H'(y,\cdot)$ on
        $[t-\ell- \frac{y_1+y_2}{2} - 1, t-\ell -
        \frac{y_1+y_2}{2}]$. Next, we apply once more
        \eqref{eq:casl=1} to all
        $ (y,t -\ell - \lfloor\frac{y_1+y_2}{2}\rfloor)$ with
        $y \in T_\ell$ such that $y_1+y_2$ is odd and deduce that
        $H(y,\cdot) \leq H'(y,\cdot)$ on
        $[t-\ell- \lfloor \frac{y_1+y_2}{2}\rfloor - 1, t-\ell -
        \lfloor \frac{y_1+y_2}{2}\rfloor]$.  We get that $H \leq H'$
        on $E_{0,t,\ell}$ which concludes the proof.
	
	\begin{figure}[htbp!]
		\begin{subfigure}{.33\textwidth}
			\centering
			\includegraphics[scale=0.4]{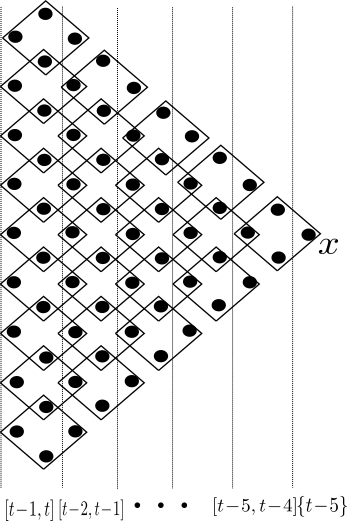}
		\end{subfigure}%
		\begin{subfigure}{.33\textwidth}
			\centering
			\includegraphics[scale=0.4]{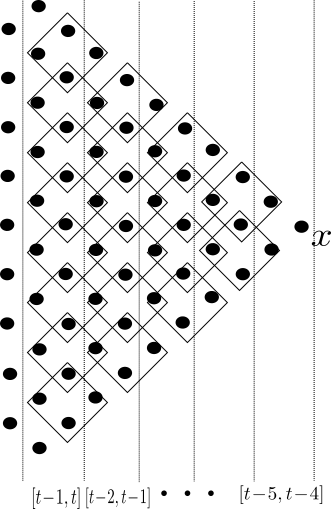}
		\end{subfigure}
		\begin{subfigure}{.33\textwidth}
			\centering
			\includegraphics[scale=0.4]{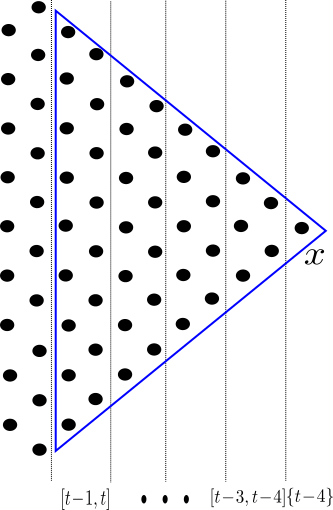}
		\end{subfigure}
		\caption{\footnotesize The induction step from $\ell$
                  to $\ell+1$.  At the bottom, we write the time
                  interval on which $H \leq H'$ for the positions
                  between the two corresponding vertical lines. The
                  inequalities in the intervals of the left figure are
                  equivalent to $H \leq H'$ on $E_{0,t,\ell+1}$ (here
                  with $\ell+1=5$). We begin with applying the case
                  $\ell=1$ in each unit square in the left figure and
                  we deduce the new time intervals on which
                  $H \leq H'$ (shown in the figure in the
                  middle). Next, we apply  the statement with $\ell=1$
                  in each unit square of the middle figure and, as
                  suggested on the right figure, we conclude that
                  $H \leq H'$ on $E_{0,t,\ell}$.}
	\label{fig:cvr}
	\end{figure}
\end{proof}

\section{Compactness}
\label{sec:compactness}

\subsection{Reducing to a simpler  initial condition}
\label{sec:remcondini}

From \cite[Lemma 2.5]{legras2017hydrodynamic} we have that, given
$f \in \bar{\Gamma}_M$ for some integer $M$ as in the statement of
Theorem \ref{theo:principalLJ} and any $L \in \N$, there exists a
natural discrete height function $h_L^f(x)$ that satisfies (a stronger
version of) \eqref{eq:hypcondinitialLJ}: it suffices to set
$h_L^f(x):=\lfloor L f(x/L)\rfloor$. Then, $h_L^f \in \Gamma_M$ and
moreover one has
\begin{equation}\label{eq:condini}
\left|\frac1L h_L^f(\lfloor x L \rfloor) - f(x)\right| \leq \frac1L \qquad 
\forall x \in \R^2,
\end{equation}
which is stronger than \eqref{eq:hypcondinitialLJ}.

A simple consequence of Corollary \ref{coro:finitespeed} is that it is
sufficient to prove Theorem 
\ref{theo:principalLJ} for the initial condition $h_L^f$. % Let us define the 
% continuous version of $B(0,R)$ defined in ~\eqref{eq:DR} by
% \begin{equation}\label{eq:Dr}
% \mathcal D(R) := \{x \in \R^2, \: |x_2-x_1|\leq R, |x_1+x_2|\leq 2R \}.
% \end{equation}
Indeed, if $(h_L)_{L 
\in \N} \in 
\Gamma_{M}^{\N}$ converges to $f$ in the sense of 
\eqref{eq:hypcondinitialLJ}, then by an immediate consequence of Corollary 
\ref{coro:finitespeed}%  (applied 
% with the maximum integer between $M$ and $M_0$)
, for all 
$T,R \geq  0$,
\begin{align*}
&\limsup_{L \to \infty} \sup_{x \in B(0,R), \, t \in[0,T]}  \left 
|\frac{1}{L} H(\lfloor Lx 
\rfloor,Lt,h_L,W) - \frac{1}{L} H(\lfloor Lx \rfloor,Lt,h_L^f,W) \right| \\
& \leq \limsup_{L \to \infty}
\sup_{x \in  B(0,R + \alpha T)} \left |\frac{1}{L} h_L(\lfloor Lx 
\rfloor) - 
\frac{1}{L} h_L^f(\lfloor Lx \rfloor) \right|\\
&  \leq \limsup_{L \to \infty} \sup_{B(0,R + \alpha T)} \left 
|\frac{1}{L} h_L(\lfloor Lx \rfloor) 
- f(x) \right| +  \sup_{B(0,R + \alpha 
T)}\left|\frac1L h_L^f(\lfloor x L 
\rfloor) - f(x)\right| =0,
\end{align*}
because of \eqref{eq:hypcondinitialLJ} and \eqref{eq:condini}. Therefore, from 
now on we will assume that $h_L=h_L^f$.
\begin{defi}
For every $f\in\bar{\Gamma}$, every realisation of the Poisson process 
$W$ and any scaling parameter $L \in \N$, we define the 
rescaled height function
\begin{equation}\label{eq:Hresc}
H_L(x,t;f,W) := \frac{1}{L} H( \lfloor L x\rfloor, Lt; h_L^f,W).
\end{equation}
\end{defi}

%\FT{Next, argue that we can restrict to the case where $f\in \bar\Gamma_A$ 
%with the set $A$  in the interior of $\mathbb T$. In fact, one can again use 
%an 
%%%approximation result with locality and monotonicity as usual \VL{CHECK it is 
%%easy for affine profile by continuity of the speed function but for other 
%%profiles ? I think we can work with $\Gamma'$}. Briefly sketch}.

%\FT{Therefore, from now on we assume that
%  \begin{eqnarray}
%    \label{eq:rb}
%  f\in \bar\Gamma_c:=\bigcup_{A\subset  
%\overset{\circ}{\mathbb{T}}}\bar\Gamma_A.  
%  \end{eqnarray}
%  }
%\`A partir de maintenant, on travaillera avec des conditions initiales dans $\bar{\Gamma}_c$ qui est plus simple que $\bar{\Gamma}$.

\subsection{Bound on temporal height differences}

The goal of this section is to obtain a control on the temporal height
differences that will be useful for showing compactness of the
sequence $(H_L(\cdot,\cdot;f,W))_{L \in \N}$.

\begin{prop}\label{prop:tconti}
	For almost every $W$, for every integer $M$, there exists a constant $C=C(M)$ 
	such that for all $x \in \R^2$, $t \geq 0$, $\delta>0$ and all $f \in 
	\bar{\Gamma}_M$,
	\begin{equation}
	\limsup_{L \to \infty} \sup_{s \in [t-\delta,t+\delta]}\left| H_L(x,t;f,W) 
	- H_L(x,s;f,W) \right| \leq C \sqrt{ t+\delta}\, \sqrt{\delta}.
	\end{equation}
\end{prop}

\begin{proof}
	Let $x \in \R^2$, $t \geq 0$, $\delta > 0$, $f \in \bar{\Gamma}_M$. Since 
	the height functions are non-increasing with time, it is enough to get an 
	upper bound on
	$$
	 H_L(x,t-\delta;f,W) 
	- H_L(x,t+\delta;f,W).
	$$  
	
	The next Lemma relates the height differences to increasing subsequences of Poisson points.
	\begin{lem}\label{lem:incseq}
          Let $x \in \Z^2$, $s,\tau \geq0$, $k\in \N$. If
          $H(x,s;h,W)-H(x,s+\tau;h,W) \geq k$ for some $h \in \Gamma$
          and $W$, then there exist increasing subsequences
          $z_0 < \cdots < z_{k-1} \leq \bar{z}(x)$, and
          $s \leq t_0 < \cdots < t_{k-1} \leq s+\tau$ such that
          $(z_i,\bar{\ell}(x),t_i) \in W$ for all $i$.
	
	For all $n\geq k$, the probability that there exists an increasing 
	sequence in $W$ as above with $z_0 > \bar{z}(x) - n$ is upper bounded by 
	$\frac{(\tau n)^k}{(k !)^2}\le \left(\frac{e^2\tau n}{k^2}\right)^k$.
	\end{lem}
	The proof is easy and is postponed to the end of the section. 

	 With Lemma \ref{lem:incseq}, Borel-Cantelli Lemma, 
	 and a rational approximation argument it is possible to show that for 
	 almost every $W$, for all $M,x,t,\delta$ as in Proposition 
	 \ref{prop:tconti}, for all $L$ large enough, for all $h \in \Gamma_M$,
	\begin{equation}
	H(\lfloor Lx \rfloor, L(t-\delta);h,\tilde{W}) - H(\lfloor Lx \rfloor, 
		L(t+\delta);h,\tilde{W}) \leq 2e \sqrt{L 2\delta \times L 
		\alpha (t+\delta)},
	\end{equation}
	with $\tilde{W}$ the restriction of $W$ to $B(\lfloor Lx\rfloor,L \alpha (t+\delta))\times [0,(t+\delta)L]$ (cf. \eqref{eq:RL}). The rest of the proof 
	follows from 
	Proposition 
			\ref{prop:finitespeed} and by setting $C(M) = 2e \sqrt{2\alpha(M)}$.

\end{proof}

	\begin{proof}[Proof of Lemma \ref{lem:incseq}]
          By definition, if $H(x,s;h,W)-H(x,t;h,W) \geq k$, then $k$
          particles crossed $x$ between time $s$ and $t$. We denote by
          $(p,\ell)$ the label of the left-most particle to the right
          of $x$ at time $s$ (i.e $\ell=\bar{\ell}(x)$ and
          $p = \min \{p, \, z_{(p,\ell)}(s) > \bar{z}(x)\}$). Since
          $k$ particles crossed $x$ in the time interval $[s,t]$,
          there exist some time $t_{k-1} \in [s,t]$ where particle
          $(p+k-1,l)$ crossed $x$ and landed at some position
          $z_{k-1} < \bar{z}(x)$. Necessarily,
          $(z_{k-1},\ell,t_{k-1}) \in W$ (since it corresponds to a
          jump) and particle $(p+k-2,\ell)$ was strictly on the left
          of $z_{k-1}$ at time $t_{k-1}$ (otherwise, the jump could not 
          have occurred). This proves the case $k=1$. If $k \geq 2$,
          since particle $(p+k-2,\ell)$ was on the right of $\bar{z}(x)$ (and $z_{k-1}$)
          at time $s$, there exists some
          $t_{k-2} \in [s,t_{k-1}]$ when it jumped strictly on the  left
          of $z_{k-1}$ and landed at some
          $z_{k-2} < z_{k-1}$ with $(z_{k-2},\ell,t_{k-2})\in W$. The proof proceeds by induction.
	 
	 The upper bound on the probability that such a sequence exists is standard 
	 (see e.g \cite[Lemma 4.1]{seppalainen1996}), so we omit it.
	\end{proof}

\subsection{Compactness for almost every realisation of the Poisson process 
\texorpdfstring{$W$}{Lg}}
The goal is to show the following:
\begin{prop}\label{prop:compacite}
For almost every realisation of $W$ the following holds: every subsequence
$(L_k)_{k \in \N}$ contains
 a sub-subsequence $(L_{k_m})_{m \in \N}$ such that for all function 
 $f \in \bar\Gamma$ one has
\begin{equation}
\forall T,R>0 \qquad\sup_{\|x\| \leq R, t \in [0,T] } \left| H_{L_{k_m}}(x,t;f,W) 
- H_{\infty}(x,t;f,W) \right| \underset{m \to \infty}{\longrightarrow} 0,
\end{equation}
for some continuous function $H_{\infty}(\cdot, \cdot; f, W) 
\in \mathcal{C}(\R^2 \times 
\R_+)$.
\end{prop}

\begin{proof}
Let us fix $W$ in the event of probability one on which Proposition 
\ref{prop:tconti} and Corollary \ref{coro:finitespeed} hold simultaneously.
 Since ${\bar\Gamma}$ is the countable union of the $\bar\Gamma_M$ and since 
 $\R^2\times \R_+$ can be written as a countable union of $[-R,R]^2\times 
 [0,T]$, modulo a standard
diagonal extraction procedure, we can restrict ourselves to the case of 
$\bar\Gamma_M$ for a fixed $M$ and of a fixed compact set $[-R,R]^2\times 
[0,T]$.

Let us start by showing the following Lemma where the function $f$ is fixed.
\begin{lem}\label{lem:compacite}
  Let $R,T>0$ and $f \in \bar{\Gamma}_M$.  For almost every
  realisation of $W$, the following holds:  every subsequence
  $(L_k)_{k \in \N}$ contains a sub-subsequence
  $(L_{k_m})_{m \in \N}$ such that
	\begin{equation}
	\sup_{\|x\| \leq R, t \in [0,T] } \left| H_{L_{k_m}}(x,t;f,W) - 
	H_{\infty}(x,t;f,W) \right| \underset{m \to \infty}{\longrightarrow} 0, 
	\end{equation}
for certain a function
  $H_{\infty}(\cdot, \cdot; f, W) \in \mathcal{C}([-R,R]^2
  \times[0,T])$.
\end{lem}
\begin{proof}
  Keeping in mind the ideas of the Arzel\`a-Ascoli Theorem, we will
  first show pointwise boundedness and asymptotic equi-continuity with respect to
  $(x,t)$.
	\begin{enumerate}
		
		\item \emph{Pointwise boundedness:}
		By Proposition \ref{prop:tconti} (with $\delta =t$), the height 
		function grows 
		at 
		most linearly i.e there exists $C'>0$ such that
		\begin{equation}
		\label{eq:gawm}
		\forall x,t \in \R^2 \times \R_+, \qquad \limsup_{L \to \infty} 
		|H_L(x,t;f,W) - 
		\underbrace{H_L(x,0;f,W)}_{\to f(x)}|  \leq C' \, t.
		\end{equation}

		\item \emph{Asymptotic equi-continuity with respect to $x,t$:} 
		Equi-continuity 
		in $x$ is 
		automatic because the spatial discrete gradients of the interface are 
		bounded 
		by 1. Thus,
		\[
		\forall x,y,s \in \R^2 \times \R^2 \times \R_+ \quad  \left| 
		H_L(x,s;f,W) - 
		H_L(y,s;f,W) \right| \leq \|x-y\| + \frac{1}{L}.
		\]
		Moreover, asymptotic equi-continuity with respect to $t$ is a direct 
		consequence of Proposition \ref{prop:tconti} and thus for all $x,t$, 
		and 
		$\delta>0$,
		\begin{equation}\label{eq:asymptxt}
		\limsup_{L\to \infty} \sup_{\substack{y, \|y-x\|\leq\delta \\ s \in 
				[t-\delta,t+\delta]}}\left| 
		H_L(y,s;f,W) - H_L(x,t;f,W) \right| \leq \delta + C \sqrt{t+\delta} 
		\sqrt{\delta}.
		\end{equation}
	\end{enumerate}

        Now, let $(L_k)_{k \in \N}$ be a subsequence. By pointwise boundedness and
        by diagonal extraction, we can find a sub-subsequence
        $(L_{k_m})_{m \in \N}$ such that for all $x,t \in \mathcal S$,
        with $\mathcal S$ a countable dense countable subset of
        $[-R,R]^2 \times [0,T]$, the sequence of real numbers
        $(H_{L_{k_m}}(x,t;f,W))_{m \in \N}$ converges to some limit to
        some $H_{\infty}(x,t;f,W)$.

 Let us extend this limit to the whole 
$[-R,R]^2 \times [0,T]$. By asymptotic equi-continuity \eqref{eq:asymptxt} and 
by density of $\mathcal S$, it is not hard to show that for any $(x,t) \in 
[-R,R]^2 \times [0,T]$, the sequence of real numbers $(H_{L_{k_m}}(x,t;f,W))_{m 
\in \N}$ is a 
Cauchy sequence and thus also converges. 
Consequently, $(H_{L_{k_m}}(\cdot,\cdot;f,W))_{m \in \N}$ converges pointwise 
on $[-R,R]^2 \times [0,T]$
to some $H_{\infty}(\cdot,\cdot;f,W)$  which is 
automatically continuous by taking the limit in \eqref{eq:asymptxt}.

It remains to show that the convergence is uniform.  This is easily
done by using compactness of $[-R,R]^2 \times [0,T]$ and asymptotic
equi-continuity \eqref{eq:asymptxt} so we omit the proof.

\end{proof}

Let us finish the proof of Proposition \ref{prop:compacite}. Since 
$\bar{\Gamma}_M$ is separable for the topology of convergence on all compact 
sets, we can find a countable dense subset that we 
call $\mathcal S'$. By Lemma \ref{lem:compacite} and diagonal extraction, from 
any 
subsequence $(L_k)_{k \in \N}$, we can extract a sub-subsequence $(L_{k_m})_{m 
\in \N}$ such that for any function $g \in \mathcal S'$, 
$(H_{L_{k_m}}(\cdot,\cdot;g,W))_{m \in \N}$ converges uniformly on 
$[-R,R]^2\times[0,T]$ to a continuous function $H_{\infty}(\cdot,\cdot;g,W) \in 
\mathcal C([-R,R]^2 \times [0,T])$.

We are going to extend this limit to any $f \in \bar{\Gamma}_M$ by showing that
$(H_{L_{k_m}}(\cdot,\cdot;f,W))_{m \in \N}$ is a Cauchy sequence in the space 
of functions from $[-R,R]^2 
\times [0,T]$ into $\R$ endowed with the 
uniform convergence which makes it complete. For 
this, we will need to use some equi-continuity with respect to $f$. It follows 
from  Corollary \ref{coro:finitespeed} and from
\eqref{eq:condini} that for all $r \geq 0$ and all $f,g \in \bar{\Gamma}_M$,
\begin{equation} \label{eq:asympcontfg}
\limsup_{L \to \infty} \sup_{\substack{x \in 
 B(0,r) \\ t \in 
		[0,T]} 
}\left| 
H_L(x,t;f,W) - 
H_L(x,t;g,W) \right| \leq \sup_{x \in B(0,r + \alpha T)}|f(x)-g(x)|.
\end{equation}
% where $\mathcal D(r)$ was defined in \eqref{eq:Dr}
 Fix $r$ 
large enough such that $B(0,r)$ contains $[-R,R]^2$ and fix $f \in 
\bar{\Gamma}_M$. Since the r.h.s of \eqref{eq:asympcontfg} can be taken 
arbitrarily small by choosing $g \in \mathcal S '$ close enough to $f$ and 
since $(H_{L_{k_m}}(\cdot,\cdot;g,W))_{m \in \N}$ is a Cauchy sequence for the 
uniform convergence on $[-R,R]^2 \times [0,T]$, this is 
also the case for $(H_{L_{k_m}}(\cdot,\cdot;f,W))_{m \in \N}$. In conclusion, 
for any $f \in \bar{\Gamma}_M$,  $(H_{L_{k_m}}(\cdot,\cdot;f,W))_{m \in \N}$ 
converges uniformly on $[-R,R]^2 \times [0,T]$ to a continuous function 
$H_{\infty}(\cdot,\cdot;f,W) \in 
\mathcal C([-R,R]^2 \times [0,T])$. This concludes the proof of Proposition 
\ref{prop:compacite}.
\end{proof}

\section{Identification of the limit}
\label{sec:identification}
All along this section, we will denote 
by 
$H_{\infty}$ any continuous limit obtained by extraction of the sequence of 
rescaled
height functions $(H_L)_{L \in \N}$ as in Proposition \ref{prop:compacite}. In 
order to 
finish 
the proof of Theorem \ref{theo:principalLJ}, we need to show that there is only 
one possible limit $H_{\infty}(\cdot,\cdot;f,W)$that coincides almost surely 
with  the unique 
viscosity solution of \eqref{eq:visc}.

\subsection{Properties of limit points}
First of all, let us show some properties of  $H_{\infty}$, inherited from those of the the microscopic 
dynamics.

\begin{prop}[Vertical translation invariance]\label{prop:invtrans}
For almost every $W$, every $f \in \bar{\Gamma}$ and $c \in \R$,
\[
H_{\infty}(\cdot,\cdot;f+c,W) = H_{\infty}(\cdot,\cdot;f,W) + c.
\]
\end{prop}
\begin{proof}
	From the definition of $h_L^f=\lfloor L f(x/L) \rfloor$, we observe 
	that $h_L^f + \lfloor Lc \rfloor \leq h_L^{f+c} \leq h_L^f + \lfloor Lc 
	\rfloor +1$. By Proposition \ref{prop:micro}, we deduce that
	\[
	H_L(\cdot,\cdot;f,W) + \frac{\lfloor Lc \rfloor}{L} \leq  
	H_L(\cdot,\cdot;f+c,W) \leq H_L(\cdot,\cdot;f,W) + \frac{\lfloor Lc \rfloor 
	+1}{L},
	\]
	which concludes the proof by taking the limit when $L$ goes to infinity.
\end{proof}

\begin{prop}[Weak locality] \label{prop:localite} For almost 
all $W$, all $R,T>0$, all integer $M$ and all $f,g \in \bar{\Gamma}_M$,
\begin{equation}
\sup_{{x \in B(0,R)}, t \in [0,T] }\left| H_{\infty}(x,t;f,W) - 
H_{\infty}(x,t;g,W) \right| \leq \sup_{x \in B(0,R+ \alpha 
T)}|f(x)-g(x)|.
\end{equation}
%with $\mathcal D (R)$ defined in \eqref{eq:Dr}.
\end{prop}
\begin{proof}
It suffices to take the limit when $L$ goes to infinity of 
\eqref{eq:asympcontfg}.
\end{proof}

Now, we are going to show a continuous version of Proposition 
\ref{prop:dlocalite}. We first have to introduce some notations. For all $x \in 
\R^2$ and $t>\delta>0$, we define $\bar{T}_{\delta}, \bar{E}_{\delta}$ and 
$\bar{E}_{x,t,\delta}$ the continuous versions of $T_\ell$, $E_\ell$ et $E_{x,t,\ell}$ by
\begin{equation}
\begin{aligned}
\bar{T}_{\delta} & := \{ y \in \R^2, \quad y_1 \leq 0, \: y_2 \leq 0 , \: -2 \delta \leq y_1 + y_2 \leq 0\}\subseteq \mathbb R^2  \\
\bar{E}_{\delta} & := \{ (y_1,y_2,-\delta - (y_1+y_2)/2), \quad y \in \bar{T}_{\delta}  \} \subseteq \mathbb R^2\times \mathbb R
\\
\bar{E}_{x,t,\delta} & := (x,t) + \bar{E}_{\delta} \subseteq \mathbb R^2\times \mathbb R_+.
\end{aligned}
\end{equation}

\begin{prop}[Space-time locality]\label{prop:localitecontinue}
For almost every $W$, every $x \in \R^2$, every $t>\delta>0$ and every $f,g \in 
\bar{\Gamma}$,
\[
\left( H_{\infty}(\cdot,\cdot;f,W) \leq H_{\infty}(\cdot,\cdot;g,W) \: \text{on } \bar{E}_{x,t,\delta} \right) \Rightarrow H_{\infty}(x,t;f,W)  \leq H_{\infty}(x,t;g,W)
\]
\end{prop}
\begin{proof}
The proof relies on Proposition \ref{prop:dlocalite}
and a continuity argument. Assume that $H_{\infty}(\cdot,\cdot;f,W) \leq 
H_{\infty}(\cdot,\cdot;g,W)$ on $\bar{E}_{x,t,\delta}$. Fix $\eps>0$.  By 
continuity of these functions and by compactness of $\bar{E}_{x,t,\delta}$, 
there exists $\eta>0$ such that
\[
H_{\infty}(\cdot,\cdot;f,W) \leq H_{\infty}(\cdot,\cdot;g,W)+\eps \: \text{on 
$\bar{E}_{x,t,\delta}^{\eta}$},
\]
where $\bar{E}_{x,t,\delta}^{\eta}$ is the set of points a distance less than 
$\eta$ from $\bar{E}_{x,t,\delta}$.
By compactness of $\bar{E}_{x,t,\delta}^{\eta}$ and by uniform convergence on 
all compact sets of the sequences $H_L(\cdot,\cdot;f,W)$ and 
$H_L(\cdot,\cdot;g+3\eps,W)$, for all $L$ large enough, the following 
inequalities hold on $\bar{E}_{x,t,\delta}^{\eta}$:
\begin{align*}
H_L(\cdot,\cdot;f,W) & \leq H_{\infty}(\cdot,\cdot;f,W) + \eps  \leq H_{\infty}(\cdot,\cdot;g,W) + 2\eps \\
& = H_{\infty}(\cdot,\cdot;g + 3\eps,W) - \eps \leq H_{L}(\cdot,\cdot;g+3\eps,W),
\end{align*}
where we used Proposition \ref{prop:invtrans} in the equality in the second 
line. Then, since we enlarged $\bar{E}_{x,t,\delta}$ by $\eta>0$, it is not 
hard to check that, for all $L$ large enough,
$$E_{\lfloor Lx \rfloor,Lt,\lfloor L\delta \rfloor} \subseteq \{(\lfloor Ly 
\rfloor,Ls), \, (y,s) \in \bar{E}_{x,t,\delta}^{\eta}\}, $$
 and thus $H_L(\cdot,\cdot;h_L^f,W) \leq H_L(\cdot,\cdot;h_L^{g+3\eps},W)$ on $E_{\lfloor Lx \rfloor,Lt,\lfloor L\delta \rfloor}$. By Proposition \ref{prop:dlocalite}, for all $\eps>0$, for all $L$ large enough, we have
\[
H(\lfloor Lx \rfloor,Lt;h_L^f,W) \leq H(\lfloor Lx \rfloor,Lt;h_L^{g+3\eps},W).
\]
Dividing by $L$ and taking the limit when $L\to \infty$ yields that for all $\eps>0$,
\[
H_{\infty}(x,t;f,W)  \leq H_{\infty}(x,t;g+3\eps,W),
\]
and the proof follows by letting $\eps$ go to $0$ and using Proposition 
\ref{prop:invtrans} again.
\end{proof}

Finally, we need to need the hydrodynamic limit in the easy case where the 
initial profile is 
linear.

\begin{prop}[Hydrodynamic limit for linear profiles]\label{prop:proline}
	For  $\rho \in \mathbb T$, we let $f_{\rho}(x):= \rho 
	\cdot x$. For almost every $W$,
	\begin{equation}\label{eq:profline}
	H_{\infty}(\cdot,t;f_{\rho},W) = f_{\rho} - v(\rho) \, t \qquad  
	\forall t \geq 0,\: \forall \rho \in \mathbb T.
	\end{equation}
\end{prop}
\begin{proof}
	Consider first   $\rho$ in the interior of $\mathbb{T}$ and let
	$$
	H^+_\infty(x,t;f_\rho,W):=\limsup_{L\to\infty}H_L(x,t;f_\rho,W)
	$$
	and analogously with the liminf for $H^-_\infty(x,t;f_\rho,W)$.  We will 
	prove that,
	for any given $x,t$, one has $W$-a.s.
	\begin{equation}\label{eq:H+bound}
	H^+_\infty(x,t;f_\rho,W)\le f_{\rho}(x) - v(\rho)t,
	\end{equation}
	(an
	analogous statement holds for $H^-_\infty(x,t;f_\rho,W)$).  Given
	this, it is then easy to deduce that \eqref{eq:profline} holds, $W$-a.s.,
	simultaneously for all $x$ and for all $t$, using the continuity of 
	$(x,t) \mapsto H_{\infty}(x,t;f_{\rho},W)$. Moreover, by continuity of 
	$\rho 
	\mapsto v(\rho)$ and $\rho \mapsto 
	H_{\infty}(x,t;f_{\rho},W)$ on $\mathbb T$  (thanks to Proposition 
	\ref{prop:localite}), we can also deduce that \eqref{eq:profline} holds 
	simultaneously for 
	all 
	$\rho \in \mathbb T$.

	% , note that it follows 
	%via Borel-Cantelli provided we
	%prove that, for every $\delta>0$, the probability
	%\begin{eqnarray}
	%\mathbb P(H_L(0,t;f_\rho,W)\ge  - v(\rho)t +\delta)
	%\end{eqnarray}
	%is summable in $L$.
	
	Let us fix $\rho \in \overset{\circ}{\mathbb T}$ and $x,t$ (we assume 
	without loss of generality that $x=0$) and let us prove 
	that  \eqref{eq:H+bound} holds $W$-as.
	First of all, we replace the initial condition
	$h_L^{f_\rho}$ by a (random) initial condition, that we call
	$h^{stat}$ (``stat'' for ``stationary''), sampled from the
	stationary measure $\pi_\rho$ with average slope $\rho$, with height
	fixed to $\delta L/2$ at position $x=0$. We recall from
	\cite{Toninelli2+1} that $\pi_\rho$ corresponds to the translation
	invariant Gibbs measure on rhombus tilings of the plane with average
	slope $\rho$ \cite{kenyon2009lectures} and that it is a time-stationary
	measure for the interface gradients. Note that the $L$-dependence of
	$h^{stat}$ is trivial: only the height offset is $L$-dependent. We
	call $H^{stat}_L(\cdot,t;\rho,W)$ the corresponding (space-time
	rescaled) height function at (macroscopic) time $t$, in analogy with
	\eqref{eq:Hresc}.  It is well known that, for any $R>0$, one has
	$\pi_\rho$-a.s.  that
	\begin{eqnarray}
	\label{eq:hstath}
	h^{stat}\ge h^{f_\rho}_L= \lfloor\rho\cdot x\rfloor \; \text{ on } \; 
	B(0,LR), \; \text{for $L$ large enough}.
	\end{eqnarray}
	Given a positive constant $C$, let us consider the localised dynamics 
	where the Poisson clocks
	outside the ball $B(0,LC)$ are turned off in the time interval
	$[0,L t]$, i.e., where $W$ is replaced by
	$\tilde W := W\cap \left(B(0,LC)\times [0,L t]\right)$. By local 
	monotonicity (more 
	precisely, apply \cite[Th. 
	5.10]{legras2017hydrodynamic}) and by \eqref{eq:hstath}, we deduce that
	$$
	 H^{stat}_L(0,t;\rho, \tilde W) \geq  H_L(0,t;f_\rho, \tilde W), \; 
	 \text{for 
	$L$ large enough}.
	$$
	Now, from Proposition 
	\ref{prop:finitespeed} 
	we have that $W$-a.s.,
	$$
	H_L(0,t;f_\rho,\tilde W)=    H_L(0,t;f_\rho, W) , \; \text{for 
		$L$ large enough},
	$$
	if $C$ is chosen large enough.
	Similarly, one has almost surely with respect to the joint law of
	$W$ and of the initial condition $h^{stat}$,
	$$
	H^{stat}_L(0,t;\rho, \tilde W)=     H^{stat}_L(0,t;\rho,W) , \; \text{for 
		$L$ large enough}.
	$$
	 This follows (again, for $C$ 
	large
	enough) from \cite[Prop. 5.13]{legras2017hydrodynamic}. In
	conclusion, we get that almost surely with respect to the joint law of
	$W$ and of $h^{stat}$,
	\begin{equation}
	 H^{stat}_L(0,t;\rho,W) \geq H_L(0,t;f_\rho,W), \; \text{for 
		$L$ large enough}
	\end{equation}
	and therefore it is enough to show \eqref{eq:H+bound} for $H^{stat}$ 
	instead of 
	$H$.
	
	By Borel-Cantelli Lemma, it is enough to prove the summability in $L$ 
	of 
	\begin{equation}
	\begin{aligned}
	& \mathbb P(H^{stat}_L(0,t;\rho,W)\ge  - v(\rho)t +\delta)\\
	& =     \mathbb 
	P(H^{stat}_L(0,t;\rho,W)- H^{stat}_L(0,0;\rho,W)\ge  - v(\rho)t +\delta/2),
	\end{aligned}
	\end{equation}
	for any $\delta > 0$  where now $\mathbb P$ is the joint law of the process 
	and of 
	the initial condition.
	On the one hand,  it follows from 
	\cite{Toninelli2+1,MR3734414} that $ - v(\rho)t$ is nothing but the 
	average growth:
	\begin{equation}
	\Esp \left[ 
	H^{stat}_L(0,t;\rho,W)- H^{stat}_L(0,0;\rho,W) \right]=  - v(\rho) t.
	\end{equation}
	On the other hand, \cite[Th. 2.2]{MR4033679} showed that $\Var \left(H^{stat}(0,t,\rho) - 
	H^{stat}(0,0,\rho)  \right)= \mathrm{O}(\log t)$ which implies that
	\begin{equation}
	\Var \left( H^{stat}_L(0,t;\rho,W)- H^{stat}_L(0,0;\rho,W)  \right) =   
	\underset{L \to \infty}{\mathrm{O}}\left( L^{-2} \log L \right)
	\end{equation}
	which is summable in $L$. The rest of the proof follows from Chebyshev's 
	inequality.
	\end{proof}

\subsection{Viscosity solution}

Let us show the following:
\begin{prop}\label{prop:soluviscosite}
For almost every $W$, the following holds: for every $f \in \bar\Gamma$, every $(x,t) \in \R^2 \times \R_{>0}$ 
and every smooth function $\varphi$ of space and time such that $H_{\infty}(x,t;f,W) = 
\varphi(x,t)$  
and $H_{\infty}(\cdot,\cdot;f,W) \leq \varphi$ (resp. $\geq \varphi$) in a 
neighborhood of $(x,t)$,
\begin{equation}\label{eq:visceq}
\text{if } \nabla \varphi(x,t) \in \mathbb T, \text{ then } 
\partial_t\varphi(x,t) + v(\nabla \varphi (x,t))  \leq  0 \quad (\text{resp.} 
\geq 0).
\end{equation}
\end{prop}
\begin{rem}
 Observe that, because of the restriction    $\nabla \varphi(x,t) \in 
 \mathbb T$, the statement  is a priori weaker than saying that 
 $H_\infty$ is a viscosity solution in the usual sense.
\end{rem}

\begin{proof}
  Suppose that $H_{\infty}(x,t;f,W) = \varphi(x,t)$,
  $H_{\infty}(\cdot,\cdot;f,W) \leq \varphi$ in a neighbourhood of $(x,t)$
  (the case $H_{\infty}(\cdot,\cdot;f,W) \geq \varphi$ is similar, so we will not
  treat it) and $\nabla \varphi(x,t) \in \mathbb T$.  Let us start by replacing 
  $\varphi$ by an affine function $\psi$ 
  by setting for all $y \in \R^2$ et $s \geq -t$
\begin{equation}
\psi(x+y,t+s) := \varphi\left(x,t\right) + y_1 \, 
\partial_{x_1} \varphi(x,t) + y_2 \, \partial_{x_2} \varphi(x,t) + 
s \, \partial_t \varphi (x,t).
\end{equation}
By Taylor expansion at order $2$, there exists some $C>0$ such that for all 
small enough $\delta$ and all $\|y\| \leq \delta$, $|s| \leq \delta$,

\begin{equation}
\left| \psi (x+y,t+s) - \varphi (x+y,t+s)\right| \leq C \delta^2 .
\end{equation}
From this inequality, we deduce that for all small enough $\delta$,
\begin{equation}\label{eq:ineqEdelta}
H_{\infty}(\cdot,\cdot;f,W) \leq \psi + C \delta^2 \quad \text{ on } 
\bar{E}_{x,t,\delta}.
\end{equation}
Now, we need the following Lemma in order to apply Proposition 
\ref{prop:localitecontinue} afterwards.
\begin{lem}\label{lem:chercherho}
There exists $\rho=\rho(x,t) \in \mathbb T$ and, for every $\delta>0$ small 
enough, there exists  $c_{\delta}=c_\delta(x,t) \in \R$ such that
\begin{equation}\label{eq:equationrhoc}
H_{\infty}(\cdot,\cdot;f_{\rho}+c_{\delta},W) = \psi + C \delta^2 \quad \text{ 
on } \bar{E}_{x,t,\delta}.
\end{equation}

The values of $\rho$ and $c_\delta$ are uniquely determined by the conditions
\begin{equation}\label{eq:systempente}
\left\{
\begin{aligned}
\rho_1 - \rho_2  & = \partial_{x_1} \varphi(x,t) - \partial_{x_2} \varphi(x,t) \\
\rho_1 + \rho_2 + v(\rho) & = \partial_{x_1} \varphi(x,t) + 
\partial_{x_2} \varphi(x,t) - \partial_t \varphi(x,t) \\
c_{\delta}   & = \varphi(x,t) - \delta \, \partial_t \varphi(x,t) - f_{\rho}(x) 
+  v(\rho) (t-\delta)+ C \delta^2.
\end{aligned}
\right.
\end{equation}
\end{lem}

Let us first admit this Lemma and conclude the proof of Proposition 
\ref{prop:soluviscosite}. By Proposition \ref{prop:localitecontinue}, 
inequality \eqref{eq:ineqEdelta} and Lemma \ref{lem:chercherho},
\begin{align*}
\varphi(x,t) &= H_{\infty}(x,t;f,W) 
 \leq H_{\infty}(x,t;f_{\rho}+c_{\delta},W) \\
& = f_{\rho}(x) + c_{\delta} - v(\rho) t&& \text{by Proposition 
\ref{prop:proline}} \\
& = \varphi(x,t) - \delta \, \partial_t \varphi(x,t) - \delta \,  v(\rho) + C 
\delta ^2 && \text{by  \eqref{eq:systempente}.}
\end{align*}
Since this holds for all small enough $\delta>0$, we finally get that
\begin{equation}\label{eq:partialtvrho}
\partial_t \varphi(x,t) + v(\rho) \leq 0.
\end{equation}
Now, combining \eqref{eq:partialtvrho} and the second equality in 
\eqref{eq:systempente},
\begin{eqnarray}
\partial_{x_1} \varphi(x,t) + 
\partial_{x_2} \varphi(x,t) \leq \rho_1 + \rho_2.  
\end{eqnarray}
Then, since $v$ is increasing with respect to $\rho_1+\rho_2$ (see Remark
\ref{rem:nablavrho}) and since 
$\rho_1-\rho_2=\partial_{x_1} \varphi(x,t) - \partial_{x_2} \varphi(x,t)$, we 
deduce that
\begin{equation}
v(\nabla \varphi(x,t)) \leq v(\rho).
\end{equation}
Finally, from \eqref{eq:partialtvrho}, we get what we wanted:
\begin{eqnarray}
\partial_t \varphi(x,t) +  v(\nabla \varphi (x,t)) \leq 0.  
\end{eqnarray}
\end{proof}

\begin{proof}[Proof of Lemma \ref{lem:chercherho}]
Let us look for a condition on $\rho$ and $c_{\delta}$ such that 
\eqref{eq:equationrhoc} is satisfied. On the one hand,
\begin{equation}
\psi\left(x+y,t-\delta - \frac{y_1 + y_2}{2} \right) = \varphi(x,t)  + 
y_1 \partial_{x_1} \varphi(x,t) + y_2  \partial_{x_2} \varphi(x,t) - \left(
\frac{y_1+y_2}{2} + \delta \right) \partial_t \varphi (x,t),
\end{equation}
and on the other hand, by Propositions \ref{prop:proline} and 
\ref{prop:invtrans},
\begin{equation}
H_{\infty}\left(x+y,t-\delta - \frac{y_1 + y_2}{2};f_{\rho}+c_{\delta},W \right) = \rho 
\cdot (x+y) + c_{\delta} - v(\rho) \left(t-\delta - \frac{y_1 + y_2}{2}\right).
\end{equation}
Therefore, it is necessary and sufficient to find $\rho \in \mathbb T$ and 
$c_{\delta} \in 
\R$ 
such that
\begin{equation}
\left\{
\begin{aligned}
\rho_1 + \frac{1}{2}v(\rho) & = \partial_{x_1} \varphi(x,t) - \frac{1}{2} 
\partial_t \varphi(x,t) \\
\rho_2 + \frac{1}{2}v(\rho) & = \partial_{x_2} \varphi(x,t) - \frac{1}{2} 
\partial_t \varphi(x,t) \\
\rho \cdot x + c_{\delta} - v(\rho) (t-\delta)  & = \varphi(x,t) - \delta \, 
\partial_t \varphi(x,t) + C \delta^2.
\end{aligned}
\right.
\end{equation}
Notice that  this 
system 
is equivalent to \eqref{eq:systempente}. We are going to show that this system 
has a unique solution $(\rho,c_{\rho})$ with $\rho \in \mathbb T$ and 
$c_{\delta} \in \R$. With the change of coordinate $\rho_h = 
\rho_1+\rho_2$, $\rho_v=\rho_1-\rho_2$ and $\hat 
v(\rho_h,\rho_v):=v((\rho_h+\rho_v)/2,(\rho_h-\rho_v)/2)$, it is equivalent to 
looking for 
$(\rho_h,\rho_v)$ such that $|\rho_v| \leq \rho_h <1$ and
\begin{equation}\label{eq:rhohv}
\left\{
\begin{aligned}
\rho_v  & = \partial_{x_1} \varphi(x,t) - \partial_{x_2} \varphi(x,t) \\
\rho_h + \hat v(\rho_h,\rho_v) & = \partial_{x_1} \varphi(x,t) + 
\partial_{x_2} \varphi(x,t) - \partial_t \varphi(x,t).
\end{aligned}
\right.
\end{equation}
The first equation fixes the value of
$\rho_v = \partial_{x_1} \varphi (x,t) - \partial_{x_2} \varphi (x,t)$
which belongs to $]-1,1[$ (since $\nabla \varphi (x,t) \in \mathbb T$
by assumption).  Now, by Remark \ref{rem:nablavrho}, for fixed
$\rho_v$, the function $\rho_h \mapsto \rho_h + \hat v(\rho_h,\rho_v)$
is strictly increasing and thus is a bijection from $[|\rho_v|,1)$ to
$[|\rho_v|,+\infty)$.  {(The reason behind this fact can be traced back to the 
  bijection between the Borodin-Ferrari dynamics and the (discretized)
  Gates-Westcott one, mentioned in Section \ref{sec:locality}).} Besides, since
$H_{\infty}(\cdot,\cdot;f,W) \leq \varphi$ on a neighbourhood of
$(x,t)$ with equality at $(x,t)$ and since
$H_{\infty}(\cdot,\cdot;f,W)$ is non-increasing w.r.t time, we deduce
that necessarily $\partial_t \varphi(x,t) \leq 0$. Moreover,
$\partial_{x_1} \varphi (x,t),\partial_{x_2} \varphi (x,t) \geq 0$
(since $\nabla \varphi (x,t) \in \mathbb T$ by assumption) and thus
$\partial_{x_1} \varphi (x,t) + \partial_{x_2} \varphi (x,t) \geq
|\partial_{x_1} \varphi(x,t) - \partial_{x_2}
\varphi(x,t)|$. Consequently, we have
$$\partial_{x_1} \varphi(x,t) + \partial_{x_2} \varphi(x,t) -
\partial_t \varphi(x,t) \in [|\rho_v|,+\infty[$$ and thus there exists
(a unique) $\rho_h \in [|\rho_v|,1)$ satisfying the second equation in
\eqref{eq:rhohv}. This finally imposes the value of $\rho$ and of
$c_{\delta}$ (by the third equation in ~\eqref{eq:systempente}).
\end{proof}

\subsection{Uniqueness and conclusion}

To conclude this section, we need to 
show that there is at most one viscosity solution in the sense of Proposition 
\ref{prop:soluviscosite}.

\begin{prop}\label{prop:uniquevisco}
	There exists a most one continuous function $u'$ on $\R^2 \times \R_+$ with 
	initial condition $f \in \bar{\Gamma}$ which is a viscosity solution of 
	\eqref{eq:visc} 
	in the sense of \eqref{eq:visceq}, 
	which satisfies the following gradient bounds:
	\begin{equation}\label{eq:gradT}
	\begin{aligned}
	0 \leq \frac{u'(x+\lambda \, e,t)-u'(x,t)}{\lambda} \leq 1 \text{ for any } 
	e\in\{(1,0),(0,1),(1,1)\}, x\in \R^2, \lambda \geq 0,
	\end{aligned}
	\end{equation}
	and which grows at most linearly 
	in time, i.e there 
	exists a finite $S>0$ such that
	\begin{eqnarray}
	\label{eq:RS}
	0\le 
	f(x)-u'(x,t)\le St \text{ for every }  x\in \R^2, t \geq 0.
	\end{eqnarray}
	Moreover, $u'$ coincides with $u$, the unique viscosity solution of 
	\eqref{eq:visc}.
\end{prop}

The proof of this Proposition is postponed to the Appendix. Since any limit 
$H_{\infty}(\cdot,\cdot;f,W)$ satisfies \eqref{eq:gradT} (by taking the limit 
in the gradient 
bounds \eqref{eq:height} \eqref{eq:height2} and \eqref{eq:grad2} satisfied by 
discrete height functions) 
and \eqref{eq:RS} 
(by taking the limit in 
\eqref{eq:gawm} and by monotonicity w.r.t time), we deduce the following 
Corollary.

\begin{coro}\label{coro:uniqueness}
		For almost every $W$ and for any limit $H_{\infty}$ as in Proposition 
		\ref{prop:compacite} we have that for every $f \in \bar\Gamma$, $(x,t) 
		\mapsto 
		H_{\infty}(x,t;f,W)$ is the unique viscosity solution of 
		\eqref{eq:visc}.
\end{coro}

To complete the proof of Theorem \ref{theo:principalLJ} it is sufficient to put together what we have obtained so far.
\begin{proof}[Proof of Theorem \ref{theo:principalLJ}]
  Let us fix $W$ in an event of probability one such that Proposition
  \ref{prop:compacite} and Corollaries \ref{coro:finitespeed} and 
  \ref{coro:uniqueness} 
  hold
  simultaneously. Let $f \in \bar{\Gamma}_{M}$. By the discussion in
  Section \ref{sec:remcondini}, without loss of generality, we can
  replace the sequence of initial height profiles $h_L$ approaching
  $f$ as in Theorem \ref{theo:principalLJ} by $h_L^f$.  Assume that
  for some $T,R>0$, the limit \eqref{eq:limitehydroGW} does not hold,
  i.e, there exists $\eps>0$ and a subsequence $(L_k)_{k \in \N}$ such
  that
		\begin{equation}\label{eq:contradiction}
		\forall k \in \N, \:
		\sup_{\|x\| \leq R, t \in [0,T]} 
		\left| H_{L_k}(x,t;f,W) - u(x,t) 
		\right| \geq \eps,
		\end{equation}
		where $u$ is the unique viscosity solution of
                \eqref{eq:visc}.  Then, by Proposition
                \ref{prop:compacite}, we can extract a sub-subsequence
                $(L_{k_m})_{m \in \N}$ such that for all
                $g \in \bar{\Gamma}$, the sequence of rescaled height
                functions $(H_{L_{k_m}}(\cdot,\cdot;g,W))_{m \in \N}$
                converges uniformly on all compact sets to some
                continuous function $H_{\infty}(\cdot,\cdot;g,W)$.  By
                Corollary \ref{coro:uniqueness},
                $(x,t)\mapsto H_{\infty}(x,t;f,W)$ has to be equal to
                $u$. This is in contradiction with
                \eqref{eq:contradiction}.
\end{proof}

\appendix

\section{Proof of Proposition \ref{prop:uniquevisco}}

The proof relies on an adaptation of a standard doubling of variables argument 
(see e.g \cite[p.64-72]{barles2013introduction}).

Let $u'$ be a continuous function initially equal to $f\in \bar{\Gamma}_M$ for 
some integer $M$ and which is a solution in the sense of Eq. \eqref{eq:visceq}
 satisfying \eqref{eq:gradT} and \eqref{eq:RS}. Let $u$ be the unique 
viscosity solution of \eqref{eq:visc}. As explained in Remark \ref{rem:compar}, 
$u(\cdot,t)$ stays in $\bar{\Gamma}_M$ for all $t$ so in particular, it 
satisfies \eqref{eq:gradT} and
 by 
the comparison principle (and since $v$ is positive on $\mathbb T_M$), it 
satisfies \eqref{eq:RS} 
with 
$S_M:=\max_{\rho \in \mathbb T_M}v(\rho)$.

Let us start by showing that $u \leq u'$.  Suppose the contrary i.e 
$u(x,t) > u'(x,t) $ for some $x,t$. Therefore, we can fix $T \geq t$ and 
arbitrary small $\eta, \beta>0$ such that the supremum
\begin{eqnarray}
\label{eq:M}
M:= \sup_{x,t \in \R^2 \times [0,T]} (u(x,t) - u'(x,t) - \eta \, t - 2\beta \, 
\|x\|^2)  
\end{eqnarray}
is positive. Now, we introduce the function
\[
\psi(x,t,y,s) := (x,t,y,s) \mapsto u(x,t) - u'(y,s) - \frac{\|x-y\|^2}{\eps^2} 
- \frac{|t-s|^2}{\alpha^2} - \eta \, t - \beta \, (\|x\|^2 + \|y\|^2),
\]
where $\eps,\alpha$ are penalisation parameters that make the supremum of 
$\psi$ looks like $M$ for small $\eps,\alpha$.

\begin{lem} \label{lem:psibar}
	Let us fix positive $T,\eta,\beta$ such that $M>0$. For any $\eps,\alpha 
	\in[0,1]$, 
	$\psi$ attains its maximum $\bar{M}$ on a point 
	$(\bar{x},\bar{t},\bar{y},\bar{s}) \in \R^2\times[0,T] \times 
	\R^2\times[0,T]$. Moreover, 
	\begin{enumerate}
		\item $0 < M \leq \bar{M}$ 
		\item $\|\bar{x}-\bar{y}\|^2 \leq C \eps^2$, $|\bar{t}-\bar{s}|^2 \leq 
		C 
		\alpha^2$ and $\|\bar{x}\|^2 + \|\bar{y}\|^2 \leq C \, \beta^{-1}$ with 
		$C 
		=2(ST+1)$
		\item $\bar{M} \underset{\eps,\alpha \to 0}{\longrightarrow}M$
		\item $\bar{t},\bar{s}>0$ for all $\eps,\alpha$ small enough.
	\end{enumerate}
\end{lem}

Let us admit this Lemma first and conclude the proof of Proposition 
\ref{prop:uniquevisco}. The function defined on $\R^2 \times [0,T]$ by
\[
(x,t) \mapsto u(x,t) - \varphi^{(1)}(x,t)
\]
has a local maximum at $(\bar{x},\bar{t})$ where
\[
\varphi^{(1)}(x,t) := u'(\bar{y},\bar{s}) + \frac{\|x - \bar{y}\|^2}{\eps^2} + 
\frac{|t - \bar{s}|^2}{\alpha^2} + \eta \, t + \beta(\|x\|^2 + \|\bar{y}\|^2).
\]

Therefore, since $u$ is solution of viscosity of \eqref{eq:visc}, and by point 
$(4)$ of Lemma \ref{lem:psibar}, we have for all $\eps,\alpha$ small enough,
\begin{equation}\label{eq:viscophi1}
\partial_t \varphi^{(1)}(\bar{x},\bar{t}) + v(\nabla 
\varphi^{(1)}(\bar{x},\bar{t})) \leq 0.
\end{equation}

\begin{rem} \label{rem:t=T}
	To be precise, we need to ensure that the viscosity inequality satisfied by 
	$u$ 
	is still valid if the local optimum on $\R^2 \times [0,T]$ is attained at 
	the border $\bar{t}=T$. This 
	is proven for example in \cite[Lemma 
	5.1]{barles2013introduction} and the same proof can be easily adapted to 
	the 
	case of 
	$u'$.
\end{rem}

Now, the function defined on $\R^2 \times [0,T]$ by
\[
(y,s) \mapsto u'(y,s) - \varphi^{(2)}(y,s)
\]
has a local minimum at $(\bar{y},\bar{s})$  where
\[
\varphi^{(2)}(y,s) := u(\bar{x},\bar{t}) - \frac{\|\bar{x} - y\|^2}{\eps^2} - 
\frac{|\bar{t} - s|^2}{\alpha^2} - \eta \, \bar{t} - \beta(\|\bar{x}\|^2 + 
\|y\|^2).
\]
In order to use the viscosity inequality \eqref{eq:visceq} satisfied by $u'$, 
we first have to 
make sure that $\nabla \varphi^{(2)}(\bar{y},\bar{s}) \in \mathbb T$. From 
\eqref{eq:gradT} and since 
$u'-\varphi^{(2)}$ has a local minimum 
at $(\bar{y},\bar{s})$, it is easy to see that $\nabla 
\varphi^{(2)}(\bar{y},\bar{s})$ is necessarily in the closure of $\mathbb 
T$. We still have to make sure that it stays away from the diagonal $\{ 
\rho_1+\rho_2 =1 \}$.

The idea is to show that 
$\nabla 
\varphi^{(2)}(\bar{y},\bar{s})$ is close to $\nabla \varphi^{(1)} 
(\bar{x},\bar{t})$ which is in $\mathbb T_M$ (because 
$u(\cdot,t)\in\bar{\Gamma}_M$ for all $t$ as said in Remark 
\ref{rem:compar} and because $u -\varphi^{(1)}$ has a local 
maximum at $(\bar x,\bar t)$). By computing the gradients, we find that

\begin{equation}\label{eq:diffgrad}
\|\nabla \varphi^{(1)} 
(\bar{x},\bar{t}) - \nabla 
\varphi^{(2)}(\bar{y},\bar{s})\| = 2 \beta \left\|\bar{x}-\bar{y}\right\| \leq 
4 
\sqrt{C\beta}
\end{equation}
where the last inequality is due to point $(2)$ of Lemma \ref{lem:psibar}. 
This shows that if $\beta$ is chosen small enough, we have that  for all 
$\eps,\alpha$ 
small enough,
\begin{equation}\label{eq:nablaphi2}
\nabla 
\varphi^{(2)}(\bar{y},\bar{s}) \in \mathbb T_{M+1} \subseteq \mathbb T.
\end{equation}
Therefore, by \eqref{eq:visceq},
\begin{equation}\label{eq:viscophi2}
\partial_s \varphi^{(2)}(\bar{y},\bar{s}) + v(\nabla 
\varphi^{(2)}(\bar{y},\bar{s})) \geq 0.
\end{equation}

Combining \eqref{eq:viscophi1} and \eqref{eq:viscophi2}, we get that for all 
$\eps,\alpha$ small enough,
\begin{equation}\label{eq:diffvisco}
\eta + v(\nabla 
\varphi^{(1)}(\bar{x},\bar{t})) - v(\nabla 
\varphi^{(2)}(\bar{y},\bar{s})) \leq 0.
\end{equation}
Finally, as noticed in Remark \ref{rem:nablavrho}, $v$ is Lipschitz on $\mathbb 
T_{M+1}$. Therefore, there exists a constant $K>0$ such that for all 
$\eps,\alpha$ small enough,
\[
\left|v(\nabla 
\varphi^{(1)}(\bar{x},\bar{t})) - v(\nabla 
\varphi^{(2)}(\bar{y},\bar{s}))\right| \leq K \, \|\nabla \varphi^{(1)} 
(\bar{x},\bar{t}) - \nabla 
\varphi^{(2)}(\bar{y},\bar{s})\| \leq K 4 
\sqrt{C\beta},
\]
because of \eqref{eq:diffgrad}. By \eqref{eq:diffvisco}, we get that
\[
\eta - 4 K \sqrt{C \beta} \leq 0
\]
which is a contradiction for $\beta$ small, since $\eta$ is strictly positive 
and $C,K$ are independent of $\eta,\beta$. We conclude that $u \leq u'$ on 
$\R^2 
\times \R_+$.

Although $u$ and $u'$ don't play symmetric roles, we omit the proof that $u' 
\leq u$ on $\R^2 \times \R_+$ since it is very 
similar.

\begin{proof}[Proof of Lemma \ref{lem:psibar}]
Let us first show that the maximum of $\psi$ is attained. We have for all 
$x,y\in \R^2$ and all $t,s \in[0,T]$,
\begin{equation}\label{eq:boundu-u'}
\begin{aligned}
u(x,t) - u'(y,s) & \leq |u(x,t) - u(y,t)| + u(y,t) - u'(y,s) \\
& \leq \|x-y\|  + u(y,0) + S \, T - u'(y,0) 
= \|x-y\| + S \, T
\end{aligned}
\end{equation}
where in the second inequality we used that $u$ satisfies \eqref{eq:gradT} 
hence is $1$-Lipschitz and is
non-increasing w.r.t time and that $u'$ satisfies \eqref{eq:RS}. We also used 
$u(y,0) = u'(y,0) = f(y)$ in the last step.
As a consequence,
\begin{align*}
\psi(x,t,y,s)  \leq \|x-y\| + S T - \beta \, (\|x\|^2 + \|y\|^2)
\end{align*}
which tends to $-\infty$ when $\|x\|,\|y\| \to + \infty$. By continuity of 
$\psi$, 
its maximum is attained at some $(\bar x,\bar t,\bar y,\bar s)$. Now, for any 
$x,t$,
\[
\psi(\bar{x},\bar{t},\bar{y},\bar{s}) \geq \psi(x,t,x,t) = u(x,t) - u'(x,t) - 
\eta t - 2 \beta \|x\|^2,
\]
which shows the first point of Lemma \ref{lem:psibar} by taking the supremum 
with respect to $x,t$.

Then, from the positivity of $\bar{M}$, we deduce that
\begin{align*}
\frac{\|\bar{x} - \bar{y}\|^2}{2\eps^2} +\frac{|\bar{t} - \bar{s}|^2}{\alpha^2} 
+ \beta \, (\|\bar{x}\|^2 + \|\bar{y}\|^2) &\leq u(\bar{x},\bar{t}) - 
u'(\bar{y},\bar{s}) - \frac{\|\bar{x} - \bar{y}\|^2}{2\eps^2} \\
& \stackrel{\text{by \eqref{eq:boundu-u'}}}\leq \|\bar{x}-\bar{y}\| + S T - 
\frac{\|\bar{x} - \bar{y}\|^2}{2\eps^2} \leq ST+1=:\frac C2 && 
\end{align*}
where the last inequality hold for all $\eps \leq 1$. This shows point $(2)$.

From point $(2)$, we know that for fixed $\beta$ and for any
$\eps,\alpha \in [0,1]$, $(\bar{x},\bar{t}),(\bar{y},\bar{s})$ stay in
a compact set so, modulo sub-sequences, we can assume that they
converge to the same point $(x_0,t_0)$ as $\alpha,\eps\to 0$, since 
$\|\bar{x}-\bar{y}\|$ and
$|\bar{t} - \bar{s}|$ tend to $0$. By continuity of $u$ and $u'$, we
have
\begin{align*}
M \leq \bar{M} &\leq u(\bar{x},\bar{t}) - u'(\bar{y},\bar{s}) - \eta \, 
\bar{t} - \beta \, (\|\bar{x}\|^2 + \|\bar{y}\|^2) \\
& \underset{\eps,\alpha \to 0}{\longrightarrow}
u(x_0,t_0) - u'(x_0,t_0) - \eta \, t_0 - 2 \beta \, \|x_0\|^2 \leq M
\end{align*}
which proves point $(3)$.

Finally, if $(x_0,t_0)$ is a common limit point of
$(\bar{x},\bar{t}),(\bar{y},\bar{s})$, then
$u(x_0,t_0) - u'(x_0,t_0) - \eta \, t_0 - 2 \beta \, \|x_0\|^2 = M$ by
the previous inequality and thus $t_0 \neq 0$ otherwise we would have
$M \leq 0$ (since $u(x_0,0) = u'(x_0,0)$) which is a
contradiction. This proves point $(4)$.

\end{proof}
{\bf Acknowledgements} This work was partially funded by
ANR-15-CE40-0020-03 Grant LSD. We are grateful to Guy Barles and Vincent Calvez for help on the literature about Hamilton-Jacobi equations.

%=======================================================================

\bibliographystyle{plain}
\bibliography{biblio}

\end{document}